\documentclass{article}

\usepackage{hyperref}
\usepackage{amssymb, latexsym, amsthm, amsmath, verbatim,graphicx}
\newtheorem{definition}{Definition}
\newtheorem*{definition*}{Definition}
\newtheorem{theorem}{Theorem}
\newtheorem{prop}[theorem]{Proposition}
\newtheorem*{theorem*}{Theorem}
\newtheorem*{prop*}{Proposition}
\newtheorem*{lemma*}{Lemma}
\newtheorem{example}{Example}
\begin{document}

\title{%
Senior Thesis for Haverford College\\
\large Convex Optimization, Newton's Method and Interior Point Method}
\author{Haoqian Li}
\date{\today}

\maketitle
\begin{abstract}

This paper consists of four general parts: convex sets; convex functions; convex optimization; and the interior point algorithm. I will start by introducing the definition of convex sets and give three common convex set examples which will be used later in this paper, then prove the significant separating and supporting hyperplane theorems. Stepping into convex functions, in addition to offering definitions, I will also prove the first order and second order conditions for convexity of a function, and then introduce couple of examples that will be used in a convex optimization problem later. Next, I will provide the official definition of convex optimization problems and prove some characteristics they have, including the existence (through optimality criterion) and uniqueness of a solution. I will also generate two convex optimization problems, one of which cannot be simply solved and requires additional skills. Afterwards I will introduce duality, for the sake of constructing the interior-point method. In the last section, I will first present the descent method and the Newton's method, which serve as the foundation of the interior-point method. Then, I will show how to use logarithmic barrier function and central path to build up the interior-point method.
\end{abstract}
\section{Introduction}
Mathematical optimization is a branch of applied mathematics. It is the selection of a best element from some set of available alternatives subject to some constraints. It has a wide range of applications in various areas including military, industry, and management. In the simplest case, an optimization problem consists of maximizing or minimizing a real function by systematically choosing input values from within an allowed set and computing the value of the function. More generally, optimization includes finding ``best available" values of some objective function given a defined domain (or input), including a variety of different types of objective functions and different types of domains. With computer programming, some large-scale optimization problems could  be solved by a computer, which consequently makes optimization problems more applicable. Linear optimization is one such field, which has been well-developed and can be well-solved with the Interior Point Method, although there are nice algorithms for modest-sized programs such as the simplex algorithm. Convex optimization is another subfield of optimization that studies how to find the minimal value of a convex functions over convex sets. It is easier to solve comparing to the general optimization since local optimal must also be global optimal, and first-order conditions are sufficient conditions for optimality. In addition, lots of optimization problems in reality are actually treated as convex optimization problems through assuming or reformulating the constraint functions and the objective function to be convex. In this paper, I'd like to focus on convex optimization and the interior-point algorithm.\\

All of the definitions are exactly stated as the ones in \cite{Boyd}. The proof for each theorem is my own work, unless otherwise stated specifically.
\section{Convex Sets}
In this paper, since we mainly investigate convex optimization problems, which seek to minimize a convex object function over a convex set of constraints (will be formally defined later), we have to define a convex set and convex functions first. The following section plays with convex sets and in the next section we will look at convex functions.
\subsection{Affine Sets}
Let's first take a glance at affine sets, which are related to convex sets.
\begin{definition*} A set $C \subseteq \mathbb R^{n}$ is affine if the line through any two distinct points in $C$ lies in $C$. 
\end{definition*}
Since for two distinct points $x_1, x_2 \in \mathbb R^n$, the points on the line passing through $x_1$ and $x_2$ can be expressed as the set of all  $y=\theta x_1+(1-\theta)x_2$, where $\theta \in \mathbb R$, thus the definition is equivalent to: if $\forall x_1, x_2\in C$ and $\theta \in \mathbb R$, we have $\theta x_1 + (1-\theta)x_2 \in C$, then $C$ is affine. 
\subsection{Convex Sets}
\begin{definition*}A set $C$ is convex if the line segment between any two points in C lies in C, i.e., if $\forall x_1,x_2\in C$ and $0\leqslant \theta \leqslant 1$, we have $\theta x_1+(1-\theta)x_2\in C$.
\end{definition*}
Note the only distinction between an affine set and a convex set is that the convex set requires the $\theta$ to be $0\leqslant \theta \leqslant 1$ instead of any arbitrary real number (it requires the line segment to be in $C$ instead of the line).
\begin{theorem*}$C$ is a convex set if and only if given arbitrary $x_1, \dotsc, x_n \in C$, $\theta_1+\dots+\theta_n=1, \theta_i \geqslant 0$ for $i=1,\dotsc, n$, we have $\theta_1 x_1+\dots+\theta_n x_n \in C$.
\end{theorem*}
\begin{proof}The proof of this theorem is offered in \cite{Boyd} p. 24- p.25.
\end{proof}
\subsection{Three important examples}
In this part, we look at three examples. The hyperplane and halfspace play important roles when we prove the separating and supporting hyperplanes theorems. Polyhedra, as they are convex sets, often serve as the constraint sets in convex optimization problems.
\begin{definition*} A hyperplane is a set of the form $\{x|a^T x=b\}$ where $a,x \in \mathbb R^n, a\neq 0$, and $b\in \mathbb R$.
\end{definition*}
Now, let's assume a point $x_0\in \mathbb R^n$ is on the hyperplane. Therefore we have $\{x|a^T x_0=b\}$. Then $a^T(x-x_0)=a^Tx-a^Tx_0=b-b=0$ for those $x$ on the hyperplane. We can then re-express the hyperplane in the following form:\begin{center}$\{x|a^T(x-x_0)=0\}=x_0+a^\perp$,
\end{center}where $a^\perp$ denotes the set of all vectors orthogonal to $a$: $a^\perp=\{v|a^Tv=0\}$. The equality above holds since the dot product of two vectors equals to 0 iff the two vectors are perpendicular to each other. Thus the hyperplane $\{x|a^T x=b\}$ can be interpreted as the set of points with a constant inner product to a given vector $a$, or as a hyperplane with normal vector $a$. These geometric interpretations are illustrated in Figure 1.\
\begin{figure}[!h]
\begin{center}\includegraphics[width=3.5 in]{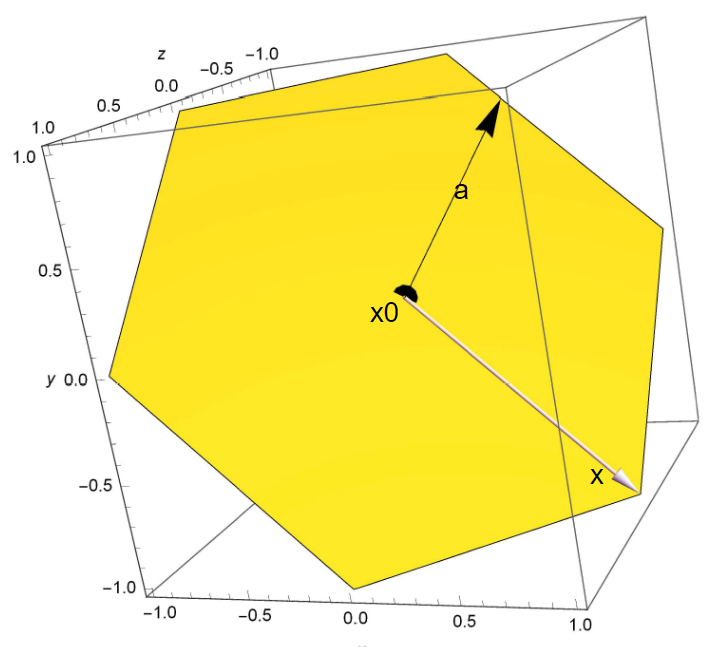}
\caption{Hyperplane $a^Tx=b$ in $\mathbb R^3$ (in this case, $a^T=[1\quad  1\quad  1]$,
$b=0$, i.e., the plane has equation $x+y+z=0$), 
with normal vector $a$ and a point $x_0$ in the hyperplane. For any point $x$ in the hyperplane, $x-x_0$ (shown as the white arrow) is orthogonal to $a$, where $x_0=0$ in this case.}
\end{center}
\end{figure}

\begin{prop*}A hyperplane is an affine set.
\end{prop*}
My own proof:
\begin{proof}Given a hyperplane $H,$ assume given arbitrary $x_1,x_2\in H$. By definition of hyperplane we have $a^T x_1=b$ and $a^T x_2=b$. Now for any $\theta$,
\begin{center}$a^T(\theta x_1+(1-\theta)x_2) = \theta a^T x_1+a^T x_2-\theta a^T x_2= \theta b+b-\theta b=b.$
\end{center}
Therefore $\theta x_1+(1-\theta)x_2\in H$ and by definition of affine sets, H is an affine set.
\end{proof}

\begin{definition*}A closed halfspace is a set of the form $\{x|a^T x\leqslant b\}$, where $a,x \in \mathbb R^n, a\neq 0$, and $b\in \mathbb R$.
\end{definition*}Note that a halfspace is just a solution set of one nontrivial linear inequality. Normally we define the halfspace in terms of $\leqslant.$ However, we can also define halfspace as a set of the form $\{x|a^T x\geqslant b\}$, where $a\neq 0$. This would be another halfspace as it is a solution set of linear inequality $\geqslant$. A hyperplane ($\{x|a^T x= b\}$, where $a\in \mathbb R^n, a \neq 0, b\in \mathbb R$) separates $\mathbb R^n$ into two halfspaces ($\{x|a^T x\geqslant b\}$ and $\{x|a^T x\leqslant b\}$), which is demonstrated in Figure 2 below. 
\begin{figure}[!h]
\begin{center}\includegraphics[width=3.5 in]{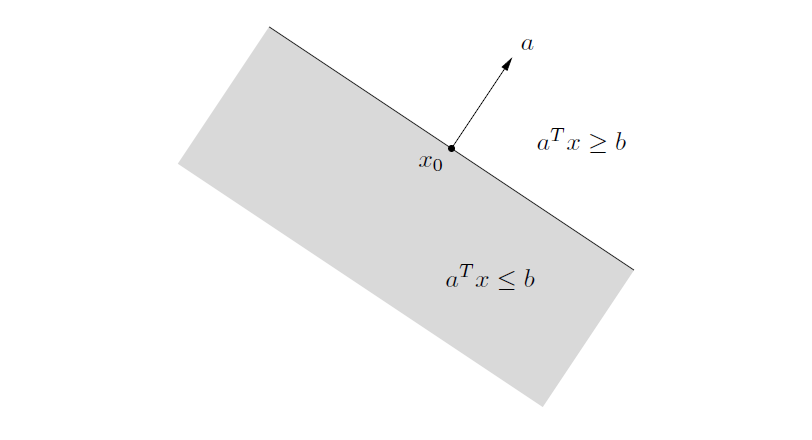}
\caption{\cite{Boyd} A hyperplane $a^Tx=b$ in $\mathbb R^2$ determines two halfspaces. The halfspace determined by $a^Tx\geqslant b$ (not shaded) is the halfspace extending in the direction $a$. The halfspace $a^Tx\leqslant b$ (shaded) extends in the direction $-a$. The vector $a$ is the outward normal of
this halfspace}
\end{center}
\end{figure}

\begin{prop*}Halfspaces are convex but not affine.
\end{prop*}
My own proof:
\begin{proof}Given a halfspace H, assume given arbitrary $x_1,x_2\in H$. By definition of halfspace we have $a^T x_1\leqslant b$ and $a^T x_2\leqslant b$. Now $\theta$ with $0 \leqslant \theta \leqslant 1$, \begin{center}$a^T(\theta x_1+(1-\theta)x_2)= \theta a^T x_1+(1-\theta)a^T x_2$. \end{center} Since $0 \leqslant \theta \leqslant 1$, we know $0\leqslant 1-\theta \leqslant 1$. Thus we have \begin{center}$\theta a^T x_1\leqslant \theta b$\end{center} \begin{center}$(1-\theta) a^T x_2\leqslant (1-\theta) b$\end{center} and thus  \begin{center}$\theta a^T x_1+(1-\theta)a^T x_2\leqslant \theta b+(1-\theta)b=b$. \end{center}
Note, however, if we don't limit $\theta$, then $a^T(\theta x_1+(1-\theta)x_2)$ could explode and won't be bounded by $b$. For a concrete counterexample, if we assume $b=2$ and  for $x_1,x_2\in H$, we have $a^T x_1=1 \leqslant b$ and $a^T x_2 = -1 \leqslant b$. Now if we take $\theta = 10$, we then would have  $\theta a^T x_1+(1-\theta)a^T x_2=10*1+(1-10)*(-1)=19>2$, which means  $\theta a^T x_1+(1-\theta)a^T x_2 \notin H$.
Thus halfspaces are not affine.\end{proof}
Similar to hyperplanes, a halfspace can also be expressed as $\{x|a^T(x-x_0)\leqslant 0\}$, where $a^T x_0=b$. Thus geometrically the hyperplane consists of an offset $x_0$ plus all vectors that make an obtuse or right angle with the vector $a$.
\begin{definition*}A polyhedron (Figure 3) is defined as the solution set of a finite number of linear equalities and inequalities: \begin{center}$P=\{x|a_i^T x\leqslant b_i, i=1,\dotsc, m, c_j^T x=d_j,j=1,\dotsc,p\}$, where $a_i, x, c_j \in \mathbb R^n, a_i, c_j\neq 0$, and $b_i, d_j\in \mathbb R$.\end{center}
\end{definition*}
A polyhedron is thus the intersection of a finite number of halfspaces and hyperplanes. Polyhedra are convex as will be proved in the following section.
\begin{figure}[!h]
\begin{center}\includegraphics[width=3.5 in]{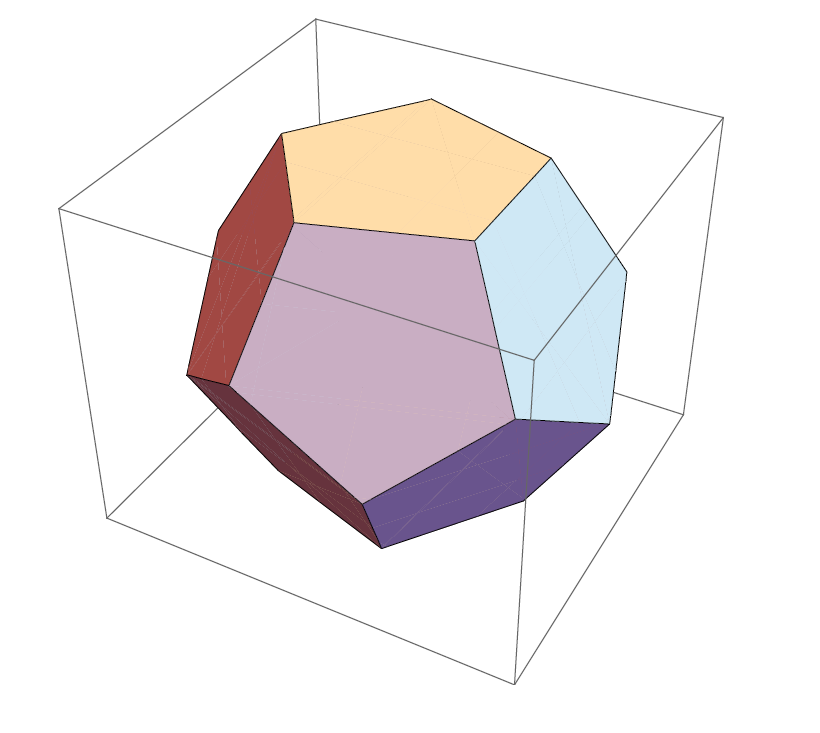}
\caption{Here is an example of polyhedron (Dodecahedron) that consists of 12 closed subspaces}
\end{center}
\end{figure}
\subsection{Operations that preserve convexity}
\begin{theorem}Convexity is preserved under intersection: let $S_a$, $a\in \Gamma$ be a collection of convex sets. Then $S=\cap_a S_a$, the intersection of these sets, is convex.
\label{thm:convexintersect}
\end{theorem}
My own proof:
\begin{proof}For any $x,y\in S, \theta \in [0,1]$, we have $x,y\in S_a$ for all $a\in \Gamma$. By convexity of sets, we know $\theta x+(1-\theta)y\in S_a$, for each $a\in \Gamma,$ where $\theta \in [0,1]$. Thus $\theta x+(1-\theta)y\in \bigcap_{a} S_a=S$. Therefore $S$ is convex.
\end{proof}
\begin{prop*}A polyhedron is convex.
\end{prop*}
\begin{proof}
Since a polyhedron is the intersection of a finite number of halfspaces and hyperplanes, which are both convex, thus polyhedra are also convex.\end{proof}
\begin{definition*}A function $f: \mathbb R^n \rightarrow \mathbb R^m$ is affine if it has the form $f(x)=Ax+b$, where $A\in \mathbb R^{m\times n}$ and $b\in\mathbb R^m$, i.e., if it is a sum of a linear function and a constant.
\end{definition*}
\begin{theorem}Suppose $S\subseteq \mathbb R^n$ is convex and $f:\mathbb R^n\rightarrow \mathbb R^m$ is an affine function. Then the image of $S$ under $f$, \begin{center}$f(S)=\{f(x)|x\in S\}$\end{center} is convex.
\end{theorem}My own proof:
\begin{proof}Given arbitrary $s_1,s_2\in f(S)$, by definition of $f(S)$, we know that there must exist $x_1,x_2\in S$ s.t. $f(x_1)=s_1, f(x_2)=s_2$. Therefore \begin{center}$Ax_1+b=s_1, Ax_2+b=s_2$. \end{center} 
Since $x_1,x_2\in S$ where $S$ is convex, by definition of convexity, we have for any $\theta \in [0,1]$, \begin{center}$\theta x_1+(1-\theta)x_2\in S$.\end{center}
Thus \begin{align*}\theta s_1+(1-\theta)s_2 &= \theta(Ax_1+b)+(1-\theta)(Ax_2+b)\\
&=\theta Ax_1+b\theta+Ax_2+b-\theta Ax_2-b\theta \\
&=A(\theta x_1+(1-\theta)x_2)+b
\end{align*}
Since we know $\theta x_1+(1-\theta)x_2\in S$, then \begin{center}$\theta s_1+(1-\theta)s_2\in f(S)$.\end{center} Therefore $f(S)$ is convex.
\end{proof}
\begin{theorem}If $S_1$ and $S_2$ are convex, then $S_1+S_2$ is convex, where $S_1+S_2=\{x+y|x\in S_1, y \in S_2\}$.
\label{thm:convexadd}
\end{theorem}
Proof is given in \cite{Boyd} on page 38.\\

We now look at couple of examples that are convex sets and will be used later in the convex optimization problem:
\begin{example} The set of points that are closer, in Euclidean norm, to (2,3,4) than the set $\{x,y,z\in \mathbb R|z\leqslant 0\}$ is a convex set. \label{example:1}
\end{example}
\begin{proof}Note in fact that the closest point to $(2,3,4)$ in the set $\{x,y,z\in \mathbb R|z\leqslant 0\}$ is $(2,3,0)$. It has Euclidean distance 4 to the point $(2,3,4)$ and thus this constraint, through geometric interpretation, would be equivalent to the set $S=B((2,3,4),4)=\{x,y,z\in \mathbb R^3| \|(x,y,z)-(2,3,4)\|_2\leqslant 2\}$. Since Euclidean balls are convex by \cite{Boyd} p.30, this set is automatically convex.\end{proof} However, we put the set in this form because it's known as a Voronoi set. We can slightly adjust this set by changing the set $\{x,y,z\in \mathbb R|z\leqslant 0\}$ into a set of given points. For example, another Voronoi set could be the set of points that are closer to (2,3,4) than the set $\{(3,4,5), (4,2,8), \cdots\}$. Thus it would be interesting to prove the Voronoi region itself being convex.\\

Here is one way to prove the Voronoi region being convex without using any geometric intuitions.
\begin{proof}
We begin by proving a proposition, which is left as the exercise 2.9 in \cite{Boyd}: \begin{prop*}Let $x_0,\cdots, x_K\in \mathbb R^n$. Consider the set of points that are at least as close as to $x_0$ as to all the $x_i$, i.e.,\begin{center}$V=\{x\in \mathbb R^n| \|x-x_0\|_2\leqslant ||x-x_i||_2, i=1, \cdots, K\}$.
\end{center}Then $V$ is convex.
\end{prop*}
\begin{proof}we have  $||x-x_0||_2\leqslant ||x-x_i||_2$ iff \begin{center}$(x-x_0)^T(x-x_0)\leqslant (x-x_i)^T(x-x_i)$\\
$x^Tx-2x_0^Tx+x_0^Tx_0\leqslant x^Tx-2x_i^Tx+x_i^Tx_i$\\
$2(x_i-x_0)^Tx\leqslant x_i^Tx_i-x_0^Tx_0$.
\end{center}
Thus we can re-express $V$ as $V=\{x|A^Tx\leqslant b\}$ where \begin{center}\[A^T=2\begin{bmatrix}
x_1-x_0\\
x_2-x_0\\
\vdots\\
x_K-x_0
\end{bmatrix}\], \[b=\begin{bmatrix}x_1^Tx_1-x_0^Tx_0\\
\vdots\\
x_K^Tx_K-x_0^Tx_0
\end{bmatrix}.\]\end{center}
In this way, we expressed $V$ in a polyhedron form and according to what we have proven before, $V$ then must be convex.
\end{proof}
The set of points closer to a given point than to a given set, i.e., $\{x |\|x-x_0\|_2\leqslant \|x-y\|_2 \text{for all } y \in S\}$ where $S\subseteq \mathbb R^n$, can be expressed as \begin{center}$\bigcap_{y\in S} \{x|\|x-x_0\|_2\leqslant \|x-y\|_2\}$.
\end{center}We just proved that for each particular $y$, the set is convex. Since the intersection of convex sets is still convex, we then know that the set described above must also be convex.\\
\end{proof}

\subsection{Separating and supporting hyperplanes}
In this section we describe the following idea: the use of hyperplanes or affine functions to separate convex sets that do not intersect. The result is the separating hyperplane theorem.
\begin{theorem}The separating hyperplane theorem: Suppose $C$ and $D$ are nonempty disjoint convex sets, i.e., $C\cap D= \emptyset$. Then there exist $a\neq 0$ and $b$ such that $a^T x-b$ is nonpositive on $C$ and nonnegative on $D$. The hyperplane $\{x|a^Tx=b\}$ is called a separating hyperplane for the sets $C$ and $D$.
\end{theorem}
The first half of the proof below is based on the ideas mentioned in \cite{Boyd} but then left as an exercise, with things reorganized and more explicit and filling gaps of the idea (the two propositions). The second half is my own proof.
\begin{proof}We first construct \begin{center}$S=\{x-y|x\in C, y\in D\}$.\end{center} We begin by claiming a simple proposition.
\begin{prop*}Scaling of a convex set is still convex: if $S\in \mathbb R^n$ is convex, $a\in \mathbb R$, then $aS=\{ax|x\in S\}$ is convex.
\end{prop*}
\begin{proof}Follows from Theorem 2.
\end{proof}
$S$ is convex by theorem \ref{thm:convexadd} since it is the sum of two convex sets, one is C, the other is -D by the above proposition. Since $C$ and $D$ are disjoint, $0 \notin S$.\\
$\textbf{Case 1}$: suppose $0$ is not in the closure of $S$, i.e, $0 \notin \textbf{cl} S$. Thus the Euclidean distance between $\textbf{cl} S$ and $\{0\}$, defined as \begin{center}$\text{inf}\{\|u-0\|_2=\|u\|_2 | u\in \textbf{cl} S, 0\in \{0\}\}$,\end{center} is positive by \cite{Analysis}, and there exists a point $s\in \textbf{cl} S$ that achieves the minimum distance. Note $s\neq 0$ since otherwise $0 \in \textbf{cl} S$. \\
Define \begin{center}$b=\frac{\|s\|_2^2}{2}$.\end{center}
We will show that the affine function \begin{center}$f(x)=s^Tx-b=s^T(x- \frac{s}{2})$\end{center} is nonnegative on $\textbf{cl} S$ and nonpositive on $\{0\}$. We prove by contradiction.\\
First suppose there exists $u\in \textbf{cl} S$, s.t. $f(u)=s^T(u-\frac{s}{2}) < 0$. We then have \begin{align*}f(u)&=s^T(u-\frac{s}{2})\\
&=s^T(u-s+\frac{s}{2})\\
&=s^T(u-s)+\frac{1}{2}\|s\|^2_2 <0 .
\end{align*}
Since $\|s\|^2_2\geqslant 0$, we must have $s^T(u-s)<0$. Now since \begin{center}$\frac{d}{dt} \|s+t(u-s)\|^2_2|_{t=0}=\frac{d}{dt} (s^T s+t^2(u-s)^T(u-s)+2ts^T(u-s))|_{t=0}=2t(u-s)^T(u-s)+2s^T(u-s)|_{t=0}=2s^T(u-s)<0$,\end{center}meaning the function $\|s+t(u-s)\|_2$ is monotone decreasing. Thus for small $t\in [0,1]$, we have $\|s+t(u-s)\|_2<\|s\|_2$,i.e., the point $s+t(u-s)$ is closer to $0$ than $s$ is. We now prove another lemma:\begin{prop*}The closure of a convex set $S$ is convex.
\end{prop*}
\begin{proof}Given arbitrary $x,y \in \textbf{cl} S$, and $\theta \in [0,1]$. By sequence interpretation of the definition of closure, there exist $x_k$, $y_k \in S$ such that $x_k \rightarrow x$ and $y_k \rightarrow y$. Since $S$ is convex, then for any $k$, we have $\theta x_k+(1-\theta)y_k\in S$. Then if we take the limit, we have $\theta x+(1-\theta)y=\lim_{k\rightarrow \infty}(\theta x_k+(1-\theta)y_k) \in \textbf{cl} S$.
\end{proof}Since $\textbf{cl} S$ is convex and contains $s$ and $u$, by convexity, \begin{center}$s+t(u-s)\in \textbf{cl} S$.\end{center}
This derives contradiction since $s$ is the point in $\textbf{cl} S$ closest to $\{0\}$. Therefore$s^Tx-b$ is nonnegative on $\textbf{cl} S$.\\
To prove $s^Tx-b$ is nonpositive on $\{0\}$ is straightfoward. For $f(x)=s^T x-\frac{\|s\|_2^2}{2}$, since the only element in the set $\{0\}$ is $0$, thus $f(x)=-\frac{\|s\|_2^2}{2}$ and consequently always nonpositive (actually always negative since $s\neq 0$). We therefore have shown $f(x)=s^Tx-b$ is nonpositive on {0} and nonnegative on $\textbf{cl}S$ since $f(x)=s^Tx-\frac{\|s\|_2^2}{2}\geqslant 0$ for all $x\in S$.\\
Now we know that for any $x\in S$, we have $f(x)=s^T(x- \frac{s}{2})\geqslant 0$. Therefore by picking $x=a-b$ where $a\in C$ and $b\in D$,\begin{align*}f(a-b)=s^T(a-b-\frac{s}{2})&\geqslant 0\\
s^Ta-s^Tb &\geqslant \frac{\|s\|^2_2}{2}
\end{align*}Since $\frac{\|s\|^2_2}{2}> 0$, we always have $s^Ta>  s^Tb$ for all $a\in C$ and $b\in D$. Therefore the set $\{s^Ta|a\in C\}$ is bounded below, which means it must have an infinimum, and this infinimum is an upper bound of the set $\{s^Tb|b\in D\}$. Let's call it $k$.  Then, define $g(x)=k-s^Tx$. This $g$ satisfies that $g$ is nonpositive on $C$ and nonnegative on $D$.
A separating hyperplane exists for two disjoint sets C and D in this case.\\
$\textbf{Case 2}$:Now assume $0\in \textbf{cl}S$. Geometrically this might happen when, for example, $C$ and $D$ are both open and their boundaries intersect. Since $0\notin S$, $0$ must lie on the boundary of S. If the interior of $S$ is the empty set, it must be contained in a hyperplane $\{s|a^Ts=b\}$. This hyperplane must include the origin, and thus $b=0$. The elements in the hyperplane, $\{s|a^Ts=0\}$, are equivalent to the elements in $\{a^Tx=a^Ty|x\in C, y\in D\}$ and is already a separating hyperplane that we want to find.\\
If the interior of $S$ is nonempty, we construct a set $S_\epsilon = \{s|B(s,\epsilon)\subseteq S\}$ where $B(s,\epsilon)$ means the ball centered at $s$ with radius $\epsilon$. 
\begin{prop*}Given arbitrary $\epsilon >0$, $\textbf{cl}S_{\epsilon} \subseteq \textbf{int}S$
\end{prop*}
\begin{proof}Given arbitrary $\epsilon$, suppose we have $x\in \textbf{cl} S_\epsilon$.\\
 If we have $x\in S_\epsilon$, then, by construction of $S_\epsilon,$ we know there exists $B(x,\epsilon)\subseteq S$. This already gives us that $x\in \textbf{int}S$ since $x$ must be away from the boundary of $S$ with $\epsilon/2$. 
If we have $x\notin S_\epsilon$ but $x\in \textbf{cl}S_\epsilon$, this means that $x$ is on the boundary/limit points of the set $S_\epsilon$. In either case, by constructin of $S_\epsilon$, for $x$ on the boundary of $S_\epsilon,$ we are able to find a point $x^*$ between $x$ and $\partial S$ as the distance between them is $\epsilon$. Since $B(x,\frac{\epsilon}{2})\subseteq B(x,\epsilon)$, and $B(x,\frac{\epsilon}{2})$ is at least $\frac{\epsilon}{2}$ away from $\partial S$, $B(x,\frac{\epsilon}{2})\subseteq S$ and therefore $x\in \textbf{int}S$. 
\end{proof}
Since $0\notin \textbf{int}S$, then $0\notin \textbf{cl}S_{\epsilon}$. Since $\textbf{cl}S_\epsilon$ is closed and convex, by what we have proven in case 1, it is strictly separated from $\{0\}$ by a hyperplane for all $s\in S_\epsilon$, say $a(\epsilon)^Ts >0$. We now aim to make $\epsilon$ small enough so that $S_\epsilon$ approaches $\textbf{int}S$. Now let $\epsilon_k, k=1,2, \cdots$ be a sequence of positive values of $\epsilon_k$ with $\lim_{k\rightarrow \infty}\epsilon_k=0$. Then for each $k,$ we would have $a(\epsilon_k)^T s>0$ for all $s\in S_{\epsilon_k}$. Since every sequence has a convergent subsequence that converges to the same limit of the sequence, we let $\bar a$ denote the limit of the subsequence of $a(\epsilon_k)$. Then $\bar a^T s>0$ for all $s\in \textbf{int}S$ since $\lim_{k\rightarrow \infty}S_{\epsilon_k}$ is $\textbf{int}S$.

Thus we have $a^T s \geqslant 0$ for all $s\in S$. Therefore we proved \begin{center}$a^T x\geqslant a^T y$ for all $x\in C, y\in D$.\end{center}
\end{proof}
\begin{definition*}Suppose $C\subseteq \mathbb R^n$, and $x_0$ is a point in its boundary $\partial C$, if there exists $a\neq 0$ satisfying $a^T x\leqslant a^T x_0$ for all $x\in C$, then the hyperplane $\{x|a^Tx=a^Tx_0\}$ is called a supporting hyperplane to $C$ at point $x_0$.
\end{definition*}
This is equivalent to saying that the point $x_0$ and the set $C$ are separated by the hyperplane $\{x|a^Tx=a^Tx_0\}$ by what we have shown above. One significant theorem, the supporting hyperplane theorem, follows from the separating hyperplane theorem. However, as it is not the focus of this paper, the details of the proof can be found in \cite{Boyd}.
\section{Convex Functions}
In this section, we investigate convex functions, another major component of convex optimization problems. We will also prove the first and second order conditions for a function to be convex, which will be used later when we optimize the convex object functions.
\begin{definition*}A function $f: \mathbb R^n \rightarrow \mathbb R$ is convex if $\textbf{dom} f$ is a convex set and for all $x, y$ in the domain of $f$, and $\theta$ with $0\leqslant \theta \leqslant 1$, we have \begin{align}f(\theta x+(1-\theta)y) & \leqslant \theta f(x)+(1-\theta)f(y).
\end{align}
\end{definition*}
A function is strictly convex if strict inequality holds in (1) whenever $x\neq y$ and $0< \theta <1$. A function is concave if $-f$ is convex, and strictly concave if $-f$ is strictly convex.\\
\subsection{Operations that preserve convexity}
\begin{theorem}If $f$ is a convex function and $c>0$, then the function $cf$ is convex. If $f$ and $g$ are both convex and have the same domain, then so is their sum $f+g$.
\end{theorem}
\begin{proof}Assume $f: \mathbb R^n\rightarrow \mathbb R$ is convex and $c>0$. Then $cf(\theta x+(1-\theta)y)\leqslant c(\theta f(x)+(1-\theta)f(y))=\theta cf(x)+(1-\theta)cf(y)$.\\
Now assume $f,g$ are both convex. Then \begin{align*}(f+g)(\theta x+(1-\theta)y) &= f(\theta x+(1-\theta)y)+g(\theta x+(1-\theta)y)\\
&\leqslant \theta f(x)+(1-\theta)f(y)+\theta g(x)+(1-\theta)g(y)\\
&=\theta(f+g)(x)+(1-\theta)(f+g)(y).
\end{align*}
\end{proof}
\begin{prop*}A function is an affine function if and only if it is both convex and concave.
\end{prop*}
My own proof:
\begin{proof}Recall a function is affine if it has the form $f(x)=Ax+b$. \\
($\Rightarrow$)Now \begin{align*}f(\theta x+(1-\theta)y)&=A(\theta x+(1-\theta)y)+b\\
&=\theta Ax+(1-\theta)Ay+b. (*)
\end{align*}
Also, \begin{align*}\theta f(x)+(1-\theta)f(y)&=\theta (Ax+b)+(1-\theta)(Ay+b)\\
&=\theta Ax+\theta b+(1-\theta)Ay+b-\theta b\\
&=\theta Ax+(1-\theta)Ay+b=*.
\end{align*}
We know that a function is convex if $f(\theta x+(1-\theta)y) \leqslant \theta f(x)+(1-\theta)f(y)$ and concave if $f(\theta x+(1-\theta)y) \geqslant \theta f(x)+(1-\theta)f(y).$ Then, because both inequalities always hold for an affine function (as we always have equality), an affine function is both convex and concave. \\
($\Leftarrow$)We begin by proving a lemma. \begin{lemma*}A function $f$ is linear if and only if for all $x,y \in \textbf{dom}f, \theta \in \mathbb R$, we have $f(\theta x+(1-\theta)y)= \theta f(x)+(1-\theta)f(y), and f(0)=0.$\end{lemma*}
\begin{proof}The foward direction is straightforward, as if $f$ is linear then by definition of linearity we have $f(\theta x+(1-\theta)y)=f(\theta x)+f((1-\theta)y)=\theta f(x)+(1-\theta)f(y)$.\\

We then focus on backward direction. Recall $f$ is linear only if given arbitrary $x_1,x_2\in \textbf{dom}f, c\in \mathbb R$, $f(cx_1)=cf(x_1)$ and $f(x_1+x_2)=f(x_1)+f(x_2)$. We first show the ``close under multiplication criterion." Suppose for all $x,y \in \textbf{dom}f, \theta \in \mathbb R$, we have $f(\theta x+(1-\theta)y)= \theta f(x)+(1-\theta)f(y)$ and f(0)=0. Given arbitrary $x_1\in \textbf{dom}f, c\in \mathbb R$, we pick $x=x_1, \theta=c, y=0$, then we have $f(cx_1+(1-c)0)=f(cx_1)=cf(x_1)+(1-c)f(0)=cf(x_1)$. We then have $f$ close under multiplication.\\
Gvien arbitrary $x_1,x_2\in \textbf{dom}f$, we pick $\theta=\frac{1}{2}, x=2x_1, y=2x_2$, then we have $f(\frac{1}{2} 2 x_1+(1-\frac{1}{2})2x_2)=f( x_1+x_2)=\frac{1}{2}f(2x_1)+(1-\frac{1}{2})f(2x_2)=\frac{1}{2}2 f(x_1)+\frac{1}{2}2f(x_2)=f(x_1)+f(x_2)$. Note in the last step we applied $f$ being closed under multiplication that we have just proven. This gives us $f$ close under addition.
\end{proof}
Recall a function is affine if it is a sum of a linear function and a constant. Let $g(x)=f(x)-f(0)$. Since $f(0)$ is simply a constant, it's sufficient to prove $g(x)$ is linear in order to prove $f$ is affine (note that if $0$ is not in the domain of $f$, one can simply twick the construction of $g$ by letting $g(x)=f(x+x_0)-f(x_0)$ for some $x_0\in \textbf{dom}f$ and everything else remains the same). By our construction and theorem 5, $g$ is both convex and concave and satisfies $g(0)=f(0)-f(0)=0$. In order to prove $g$ is linear, it's sufficient to prove \begin{center}$g(\theta x+(1-\theta)y)= \theta g(x)+(1-\theta)g(y).$ (*)
\end{center} holds for all $\theta\in \mathbb R$.
Since the function $g$ is both convex and concave, given arbitrary $x,y\in \textbf{dom} g$, and $\theta \in [0,1]$, we have \begin{center}$g(\theta x+(1-\theta)y)= \theta g(x)+(1-\theta)g(y).$ (*)
\end{center}
We already finished proving for $\theta \in (0,1)$. Now suppose $\theta >1$. Consider $a=\theta x+(1-\theta)y$. If we solve $x$ in terms of $a$ and $y$, we get \begin{center}$x=\frac{1}{\theta} a+(1-\frac{1}{\theta})y$.\end{center}Since $\theta>1$, we then have $\frac{1}{\theta} \in [0,1]$. Thus by definition of convexity and concavity of $g$, we have \begin{center}$g(x)=g(\frac{1}{\theta} a+(1-\frac{1}{\theta})y)=\frac{1}{\theta} g(a)+(1-\frac{1}{\theta})g(y)$.
\end{center}
Rearrange this equation and substitute $a=\theta x+(1-\theta)y$ back, we then get \begin{center}$g(\theta x+(1-\theta)y)= \theta g(x)+(1-\theta)g(y)$ for $\theta >1$.\end{center}
Now suppose $\theta <0$. Still consider $a=\theta x+(1-\theta)y$. This time we solve $y$ in terms of $x$ and $a$. We get \begin{center}$y=kx+(1-k)a$ where $k=\frac{\theta}{\theta -1}$.\end{center} Note since $\theta <0$, we have $k=\frac{\theta}{\theta-1} \in [0,1]$. The rest part of the proof follows the similar logic as in the case $\theta >1$. Therefore we proved \begin{center}$g(\theta x+(1-\theta)y)= \theta g(x)+(1-\theta)g(y)$ for all $\theta \in \mathbb R.$\end{center}
\end{proof}
\subsection{First-order conditions}
\begin{theorem}
Suppose $f$ is differentiable. Then $f$ is convex if and only if its domain is convex and \begin{align} f(y) &\geqslant f(x)+\nabla f(x)^T(y-x) \end{align} holds for all $x,y$ in the domain of $f$.
\label{thm:firstorder}
\end{theorem}
My own proof:
\begin{proof}$(\Rightarrow)$: Assume $f$ is convex. By definition, we then have its domain a convex set and for all $x, y \in \textbf{dom}f$, and $\theta$ with $0\leqslant \theta \leqslant 1$, we have \begin{center} $f(\theta x+(1-\theta)y) \leqslant \theta f(x)+(1-\theta)f(y)$.\end{center}
Now we arrange the equation above: \begin{align*}f(\theta x+(1-\theta)y) &\leqslant \theta f(x)+(1-\theta)f(y)\\
f(y+\theta (x-y)) &\leqslant \theta(f(x)-f(y))+f(y)\\
f(y)+\frac{f(y+\theta(x-y))-f(y)}{\theta} &\leqslant f(x).
\end{align*}
We now define $g(\theta)=f(y+\theta(x-y))$. Then $g(0)=f(y)$ and $g$ is differentiable since it consists of linear composition of $f$. If we use $g$ to express the inequality above we have \begin{center}$f(y)+\frac{g(\theta)-g(0)}{\theta}\leqslant f(x)$.\end{center}Since $g$ is differentiable, we can take limit as $\theta \rightarrow 0$ and get \begin{center}$f(y)+\lim\limits_{\theta \to 0} \frac{g(\theta)-g(0)}{\theta}\leqslant f(x)$.\end{center}
\begin{center}$f(y)+g^\prime (0)\leqslant f(x)$.\end{center}
Since $g' (\theta)=(\nabla_y f(y+\theta(x-y)))^T (x-y)$, we then have \begin{center}$g' (0)=(\nabla_y (f(y))^T(x-y)$.\end{center}
Therefore substitute this back we then have \begin{center}$f(x)\geqslant f(y)+\nabla f(y)^T(x-y)$.\end{center} Since this inequality holds for any $x,y$ in the domain of $f$, we can simply switch $x$ and $y$, which corresponds with the theorem.\\
$(\Leftarrow)$: Define $z=\theta x+ (1-\theta)y$ where $\theta \in [0,1]$. Since $x,y$ in the domain of $f$ and the domain is convex, $z$ must also be in the domain of $f$ by convexity of sets. Since we also have the inequality in (2) for any points in the domain of $f$, we then have the following inequality: \begin{center}$f(x)\geqslant f(z)+ \nabla f(z)^T (x-z)$.\end{center}
Since $\theta \geqslant 0$, we can multiply both sides by $\theta$ and get \begin{align}\theta f(x) &\geqslant \theta f(z)+\theta \nabla f(z)^T (x-z).\end{align}
Similarly, we also have \begin{center}$f(y)\geqslant f(z)+ \nabla f(z)^T (y-z)$.\end{center} Since $\theta \leqslant 1$, we can multiply both sides by $(1-\theta)$ and get \begin{align}(1-\theta) f(y) &\geqslant (1-\theta) f(z)+(1-\theta) \nabla f(z)^T (y-z).\end{align}
Adding (3) and (4) we then have \begin{align*}\theta f(x)+(1-\theta)f(y) &\geqslant \theta f(z)-\theta f(z)+f(z)+\theta \nabla f(z)^T (x-z)+ \nabla f(z)^T (y-z)-\theta \nabla f(z)^T (y-z)\\
&= f(z)+\nabla f(z)^T (\theta (x-z)+(y-z)-\theta (y-z))\\
&= f(z)+\nabla f(z)^T (\theta x+(1-\theta)y-z)\\
&= f(z)+\nabla f(z)^T (z-z)\\
&=f(z)=f(\theta x+(1-\theta)y)
\end{align*}We thus finished proving both directions.
\end{proof}
\subsection{Second-order conditions}
\begin{theorem}Assume $f$ is twice differentiable, then $f$ is convex if and only if its domain is convex and for all $x\in \textbf{dom}f$, we have \begin{center}$\nabla^2 f(x)\geqslant 0$.\end{center}
\label{thm:secondorder}
\end{theorem}
$\nabla^2 f(x)$ here stands for the second derivative (if the domain of $f$ is in $\mathbb R$) or the Hessian matrix (domain of $f$ is in $\mathbb R^n$)of $f$, and "$\geqslant$" means that the Hessian matrix is positive semidefinite.\\

My own proof:
\begin{proof}We first consider the case $f: \mathbb R \rightarrow \mathbb R$.\\
($\Rightarrow$): Suppose $f$ is convex. Given arbitrary $x,y \in \textbf{dom}f$ with $y>x$, by first order condition we have \begin{align*}f(y)&\geqslant f(x)+f' (x) (y-x)\\
f(x)&\geqslant f(y)+f'(y) (x-y).\end{align*}
Rearrange the two inequalities we have  \begin{align*}f(y)-f(x)&\geqslant f' (x) (y-x)\\
f(y)-f(x)&\leqslant f'(y) (y-x).\end{align*}
Therefore \begin{center}$f' (y)(y-x)\geqslant f(y)-f(x)\geqslant f' (x)(y-x)$.\end{center}
\begin{center}$\rightarrow f' (y)(y-x)-f'(x)(y-x)\geqslant 0$.\end{center}
Since $y>x$, dividing both sides by $(y-x)^2$ we have \begin{center}$\frac{f' (y)-f' (x)}{y-x}\geqslant 0$.\end{center}
Let $y\rightarrow x$, we get \begin{center}$f^{''} (x) \geqslant 0,  \forall x\in \textbf{dom}f$.\end{center}
($\Leftarrow$): Suppose $f^{''} (x)\geqslant 0$ for all $x\in \textbf{dom}f$. Given arbitrary $x,y$ in $\textbf{dom}f$ with $x\leqslant y$, by the Taylor's theorem, for some $z\in [x,y]$, we have \begin{center}$f(y)=f(x)+f^{'}(x)(y-x)+\frac{1}{2}f^{''}(z)(y-x)^2$. 
\end{center}
Since $y-x\geqslant 0$ and by assumption $f^{''}(z)\geqslant 0$, we have \begin{center}$f(y)\geqslant f(x)+f^{'} (x)(y-x)$.\end{center}
By the first-order condition of convexity we know $f$ is then convex.\\
We proved the theorem for $f:\mathbb R \rightarrow \mathbb R$. We now generalize this to higher dimension, where $f:\mathbb R^n \rightarrow \mathbb R$. We start by proving a lemma.
\begin{prop*}A function $f:\mathbb R^n \rightarrow \mathbb R$ is convex if and only if the function $g:\mathbb R\rightarrow \mathbb R$, given by $g(t)=f(x+ty)$ is convex, for all $x$ in domain of $f$, all $y\in \mathbb R^n$.
\end{prop*}
\begin{proof}
($\Rightarrow$)We prove by contrapositive. Suppose for some $x,y$, $g(t)$ is not convex. Then there exists $\theta\in [0,1], t_1, t_2$ such that $g(\theta t_1+(1-\theta)t_2)>\theta g(t_1)+(1-\theta)g(t_2)$. Therefore \begin{align*} f(x+(\theta t_1+(1-\theta)t_2)y)&=f(\theta(x+t_1 y)+(1-\theta)(x+t_2 y))\\
&>\theta f(x+t_1y)+(1-\theta)f(x+t_2y),\end{align*} which then implies $f$ not convex.\\
($\Leftarrow$). Suppose for $x\in \textbf{dom}f$ and $y\in \mathbb R^n$, we have $g(t)=f(x+ty)$ is convex. Then $g(\theta t_1+(1-\theta)t_2)\leqslant \theta g(t_1)+(1-\theta)g(t_2)$. We then have \begin{align*}
f(x+(\theta t_1+(1-\theta)t_2)y)&= f(\theta(x+t_1 y)+(1-\theta)(x+t_2 y))\\
&\leqslant \theta f(x+t_1y)+(1-\theta)f(x+t_2y),
\end{align*} which implies $f$ convex.
\end{proof}
Thus by our lemma, $f:\mathbb R^n\rightarrow \mathbb R$ is convex if and only if $g(t)=f(x_0+tv)$ is convex for all $x_0\in \textbf{dom}f$ and $v\in \mathbb R^n$. By our proof for the base case, since $g: \mathbb R\rightarrow \mathbb R$, we know $g(t)$ is convex if and only if $g^{''}(t)=v^T\nabla^2 f(x_0+tv)v \geqslant 0$ (by chain rule). Then we have $f$ convex if and only if \begin{center}$v^T\nabla^2 f(x_0+tv)v\geqslant 0$.\end{center} for all $x_0\in \textbf{dom}f, v\in \mathbb R^n$ and $\{t|x_0+t v \in \textbf{dom}f\}$.\\
In other words, $f$ is convex if and only if the Hessian matrix of $f$ is positive semi-definite for all $x\in \textbf{dom}f$. 
\end{proof}
Below is an example using the second order condition in order to prove a set being convex.
\begin{example}$\{x\in \mathbb R^n_+| \Pi^n_{i=1} x_i\geqslant 1\}$ is convex. \label{example:product}
\end{example}
\begin{proof}
We begin by first proving a proposition.
\begin{prop*}If $a,b\geqslant 0$ and $\theta \in [0,1]$, then $a^\theta b^{1-\theta}\leqslant \theta a+(1-\theta)b$.
\end{prop*}
\begin{proof}If we define $f(x)=-\log x: \mathbb R^+ \rightarrow \mathbb R$, we can see that $f^{''} (x)=\frac{1}{x^2}>0$ for all $x\in \mathbb R^+$. This by the second order condition of convexity, tells us that $f(x)=-\log x$ is convex. We then have \begin{center}$f(\theta x+(1-\theta)y)\leqslant \theta f(x)+(1-\theta)f(y)$\end{center} or equivalently \begin{center}$-\log(\theta x+(1-\theta)y)\leqslant -\theta \log(x)-(1-\theta)\log(y).$\end{center}
for $\theta\in [0,1]$. Since $\theta \log(x)+(1-\theta)\log(y)=\log (x^\theta y^{1-\theta})$, we have \begin{center}$\log(\theta x+(1-\theta)y)\geqslant \log (x^\theta y^{1-\theta}).$\end{center}Since the function $f(x)=\log (x)$ is monotone increasing, we then have \begin{center}$\theta x+(1-\theta)y\geqslant x^\theta y^{1-\theta}$.\end{center}
\end{proof}
We now prove the theorem. Assume that $\Pi_ix_i\geqslant 1$ and $\Pi_i y_i\geqslant 1$. Using the inequality proven above, we have \begin{center}
$\Pi_i(\theta x_i+(1-\theta)y_i)\geqslant \Pi x_i^\theta y_i^{(1-\theta)}=(\Pi_i x_i)^\theta(\Pi_i y_i)^{(1-\theta)}\geqslant 1 \text{for } \theta\in [0,1] .$\end{center}
\end{proof}Note that the "1" on the right-hand side of the equality can actually be any number larger than 1.
\begin{example}Given $a\in \mathbb R^n$, the function $||a-x||_2$ with domain $\mathbb R^n$ is convex.
\label{example: norm}
\end{example}
\begin{proof}
We then seek to prove that given $x,y\in \mathbb R^n$, for $\theta \in [0,1],$ we have \begin{center}$||a-(\theta x+(1-\theta)y)||_2\leqslant \theta ||a-x||_2+(1-\theta)||a-y||_2$.
\end{center}
Now \begin{align*}||a-(\theta x+(1-\theta)y)||_2 &= ||\theta a+(1-\theta)a-(\theta x+(1-\theta)y)||_2\\
&= ||\theta (a-x)+(1-\theta)(a-y)||_2\\
&\leqslant ||\theta(a-x)||_2+||(1-\theta)(a-y)||_2 \\
&= \theta ||a-x||_2+(1-\theta)||a-y||_2.
\end{align*}
Note in the last two steps we applied the triangle inequality of vector norms and the homogeneity of norm functions, repectively.
\end{proof}
\section{Convex Optimization}
In this section, we officially introduce our main focus, the convex optimization problems, and some features of the solution of the problem, including optimality criterions under different circumstances, such as with equality constraints alone, or unconstrained.
\subsection{Convex Optimization Problems}
\subsubsection{Standard form}
A convex optimization problem has the form \begin{align*}&\text{minimize} &f_0(x)\\
&\text{subject to} &f_i(x)\leqslant 0, i=1, \cdots, m\\
& & h_i(x)=0, i=1,\cdots, p,\end{align*} where $f_0, \cdots, f_m$ are convex functions and $h_i(x)$ are affine. We call $x\in \mathbb R^n$ the $\textit{optimization variable}$ and the function $f_0: \mathbb R^n\rightarrow \mathbb R$ the $\textit{objective function}$ or $\textit{object function}$. The inequalities $f_i(x)\leqslant 0$ are called $\textit{inequality constraints}$, and the corresponding functions $f_i: \mathbb R^n\rightarrow \mathbb R$ are called the $\textit{inequality constraint}$ $\textit{functions}$. The functions $h_i(x)$ are called $\textit{equality constraint functions}$ and are affine, i,e., has form $h_i(x)=a_i^T x-b_i$. The set of points for which the objective function and all constraint functions are defined is called the $\textit{domain}$ of the optimization problem. A point in the domain is defined as $\textit{feasible}$ if it satisfies the constraints $f_i(x)\leqslant 0, i=1,\cdots, m$ and $h_i(x)=0, i=1,\cdots, p$. The problem is said to be feasible if there exists at least one feasible point, and infeasible otherwise. The set of all feasible points is called the $\textit{feasible set}$ or the $\textit{constraint set}$.\\
The $\textit{optimal value p*}$ of the problem is defined as \begin{center}$\text{p*=inf}\{f_0(x)|f_i(x)\leqslant 0, i=1, \cdots, m, h_i(x)=0, i=1,\cdots, p\}$.\end{center}
\begin{prop*}The feasible set of a convex optimization problem is convex.
\end{prop*}
\begin{proof}Since the feasible set is the set of all feasible points, it is the intersection of the domain of the problem $D=\bigcap_{i=0}^m \textbf{dom}f_i$, which is convex, with m convex sublevel sets $\{x\|f_i(x)\leqslant 0\}$ and p hyperplanes $\{x\|a_i^Tx=b_i\}$ by theorem \ref{thm:convexintersect}. Thus the feasible set must be convex.
\end{proof}
Thus, in convex optimization, we minimize a convex objective function over a convex set.
\begin{definition*}x* is an optimal point if x* is feasible and $f_0(x^*)=p^*$.  A feasible point $x$ is locally optimal if there is an $R>0$ such that \begin{center}$f_0(x)=\text{inf}\{f_0(z)|f_i(z)\leqslant0, i=1,\cdots, m, h_i(z)=0, i=1,\cdots ,p, ||z-x||_2\leqslant R\}$\end{center}
\end{definition*}
\begin{theorem}For convex optimization problems, any locally optimal point is also a global optimal.
\end{theorem}
The ideas are based on \cite{Boyd}(p.138) with details being filled and things rearranged.
\begin{proof}Suppose that $x$ is locally optimal for a convex optimization problem. Then by definition, there exists $R>0$ such that \begin{center}$f_0(x)=\text{inf}\{f_0(z)|z \text{ feasible}, ||z-x||_2\leqslant R\}.$\end{center}
Assume $x$ is not a global optimal, then there must exist $y$ that is feasible and satisfies $f_0(y)<f_0(x)$. Since $x$ is a local optimal, we must have $||y-x||_2>R$. By convexity of the feasible set, we have \begin{center}$(1-\theta) x+\theta y$\end{center} also in the feasible set for $\theta=\frac{R}{2||y-x||_2}\in [0,1]$. Then, by convexity of $f_0$ we have \begin{align}f_0((1-\theta) x+\theta y)&\leqslant (1-\theta) f_0(x)+\theta f_0(y)\nonumber \\
&<(1-\theta) f_0(x)+\theta f_0(x)\nonumber \\
&=f_0(x)\end{align}
However, note that 
\begin{center}$||(1-\theta)x+\theta y-x||_2=||\theta(y-x)||_2=\frac{R}{2||y-x||_2} ||y-x||_2=\frac{R}{2}<R$
\end{center}which means that $(1-\theta)x+\theta y$ lies in the neighborhood of $x$ in the local sense and thus we must have $f_0(x)\leqslant f_0((1-\theta) x+\theta y)$ by definition of local optimal. We then derive contradiction to (5).
\end{proof}
\begin{theorem}For a convex optimization problem where $f_0$ is strictly convex on the convex feasible set $D$, the optimal solution, if it exists, must be unique.
\end{theorem}
My own proof:
\begin{proof}Suppose there are two optimal solutions $x,y$. This means that $x,y\in D$ and \begin{center}$f_0(x)=f_0(y)\leqslant f(z), \forall z \in D$\end{center}
By convexity of $D$, we know that $\frac{1}{2} x+\frac{1}{2} y\in D$. By strict convexity of $f_0$, we have \begin{align*}f_0(\frac{1}{2} x+\frac{1}{2} y)&<\frac{1}{2}f_0(x)+\frac{1}{2}f_0(y)\\
&=\frac{1}{2}f_0(x)+\frac{1}{2}f_0(x)\\
&=f_0(x)\end{align*}
Thus we have $f_0(\frac{1}{2} x+\frac{1}{2} y)<f_0(x)$, which contradicts the definition of optimal solutions.
\end{proof}
\subsubsection{An optimality criterion for differentiable $f_0$}
\begin{theorem}Suppose that the objective $f_0$ in a convex optimization problem is differentiable. Let $D$ denote the feasible set. Then $x$ is optimal if and only if $x\in D$ and \begin{center}$\nabla f_0(x)^T(y-x)\geqslant 0, \forall y\in D$. 
\end{center}\label{thm:optimality}
\end{theorem}
My own proof:
\begin{proof}($\Leftarrow$) Suppose $x\in D$ satisfies \begin{center}$\nabla f_0(x)^T(y-x)\geqslant 0, \forall y\in D$. \end{center}
By the first order condition of convexity, we have\begin{center}$f_0(y)\geqslant f_0(x)+\nabla f_0(x)^T(y-x), \forall y\in D$.\end{center}
Thus \begin{center}$f_0(y)-f_0(x)\geqslant \nabla f_0(x)^T(y-x)\geqslant 0, \forall y\in D.$
\end{center}$x$ is then by definition an optimal solution.\\
($\Rightarrow$)Suppose $x$ is optimal but for some $y\in D$ we have \begin{center}$ \nabla f_0(x)^T(y-x)<0.$\end{center}
Now consider $z(t)=ty+(1-t)x=x+t(y-x)$, where $t\in [0,1]$. By definition of convexity of the feasible set, $z(t)\in D$. Now Define $h(t)=f_0(z(t))=f_0(x+t(y-x))$. We then have
\begin{align*}
h'(t)&=(y-x)^T \nabla f_0(x+t(y-x))\\
\rightarrow h'(0)&=(y-x)^T \nabla f_0(x)<0.
\end{align*}
This tells us that when we move away from $0$ by a small positive distance $t$, we would get $h(t)<h(0)$. This is equivalent to \begin{center}
$f_0(z(t))<f_0(x)$ for a small positive $t$. 
\end{center}
This contradicts with $x$ being optimal, and thus we have $\nabla f_0(x)^T(y-x)\geqslant 0, \forall y\in D$.
\end{proof}
Here are two special cases derived from this theorem, which will be used later.
\begin{theorem}For an unconstrained problem, the condition reduces to \begin{center}$\nabla f_0(x)=0$\end{center} for $x$ to be optimal.
\label{thm:unconstrained}
\end{theorem}
\begin{proof}
($\Rightarrow$) If $\nabla f_0(x)=0$, then $\nabla f_0(x)^T(y-x)= 0, \forall y\in D$.\\
($\Leftarrow$) Suppose $x$ is optimal, for all feasible $y$ we have by previous theorem $\nabla f_0(x)^T(y-x)\geqslant 0$. Since there is no constraint and $f_0$ differentiable, all $y$ are feasible. Let $y=x-t \nabla f_0(x)$, where $0<t\in \mathbb R$ is a parameter. Thus we have \begin{center}$\nabla f_0(x)^T(y-x)=-t||\nabla f_0(x)||^2_2\leqslant 0.$\end{center}
We then must have $\nabla f_0(x)=0$.
\end{proof}
\begin{theorem}For a convex optimization problem that only has equality constraints, i.e., min$f_0(x)$ subject to $Ax=b$, where $A\in \mathbb R^{m\times n}$, $x^*$ is optimal if and only if there exists $k\in \mathbb R^m$ such that \begin{center}$\nabla f_0(x^*)+A^T k=0,$ where $Ax^*=b$ (i.e., $x^*$ feasible)\label{eqaulity constraint}\end{center}
\end{theorem}
Proof of this theorem is given on \cite{Boyd} p.141.
We now look at two examples. The first example contains the application of theorem 11 in three dimensional space. The second example lies in two dimensional space and involve the Chebyshev Center. It includes both rewriting the problem into a convex optimization problem and it proposes a circumstance in which Theorem \ref{thm:optimality} can't be used as a practical algorithm.\\
\begin{example}
\begin{align*}\text{minimize } & x^2+y^2+z^2-xy-yz\\
\text{subject to } &\text{the set of points that are closer, in Euclidean norm, to (2,3,4) than to} \\
&\text{the set } \{x,y,z\in \mathbb R|z\leqslant 0\} \text{and must obey };\\
& xyz\geqslant 1 \text{and };\\
& x,y,z\geqslant 0;
\end{align*}
\label{exa:object}
\end{example}
We first aim to show that the problem is indeed a convex optimization problem.
\begin{proof}The Hessian matrix of the object function is 
\[
\begin{bmatrix}
2 &-1 &0\\
-1 & 2 &-1\\
0 & -1 &2
\end{bmatrix}
\]
We know that the Hessian matrix is positive semidefinite (by computing eigenvalues, which are all nonnegative), thus the object function must be convex by theorem \ref{thm:secondorder}, the second order condition of convex functions,.\\

The first constraint, by what we have proven before in example \ref{example:1}, is convex.
Since constraint two is only a special case of Example \ref{example:product}, it certainly has to be convex. The last constraint is also convex since it is the intersection of three halfspaces, $x\geqslant 0$, $y\geqslant 0$, $z\geqslant 0$ and halfspace is convex. We then proved that this problem is a convex optimization problem.
\end{proof}
This problem, however, can be expanded in many different ways based on the two propositions that we have proven above. For example, the set in the first constraint can be much more complicated. For example, we can set the constraint to be the set of points that are closer to a series of points, which would still be a polyhedron according what we have proven. We can also introduce more variables and use Hessian matrix to test the convexity of the object function, or introduce more relationships similar to our second constraint, etc. Thus we can call this simple example as a prototype, and we will continue use it for the sake of simplicity of calculations.

\begin{example}Suppose that we are given $k$ points $a_1,a_2,\cdots, a_k \in \mathbb R^n$. The objective is to find the center of the minimum radius closed ball containing all the points. The ball is called the Chebyshev ball and the corresponding center is the Chebyshev center. The problem can be written as \begin{align*}&\text{minimize } &r\\
&\text{such that there exists } x \text{such that } &a_i\in B[x,r], i=1,2,\cdots, k,
\end{align*}where $B[x,r]$ denotes the closed ball with radius $r$ and tbe center $x$.\label{example:2}
\end{example}
Now since the ball $B[x,r]$ can be rewritten as $\{y: ||y-x||_2 \leqslant r\}$, we can reformulate the problem to \begin{align*}\text{minimize } &f_0(x,r)=r\\
\text{subject to } & ||a_i-x||_2\leqslant r, i=1,2,\cdots, k,
\end{align*}where $f_0:\mathbb R^n\times \mathbb R\rightarrow \mathbb R$. We prove that it's a convex optimization problem.
\begin{proof}The object function is obviously convex since it is apparently affine and a function is affine if and only if it is both convex and concave (proven before.)\\
We now seek to prove $||a_i-x||_2-r, i=1,2,\cdots, k$ are convex functions. Since $||a_i-x||_2$ is convex by Example \ref{example: norm}, $r$ is affine, it follows that $||a_i-x||_2-r, i=1,2,\cdots, k$ are convex functions since addition preserve convexity.
\end{proof}
We can see that our example is indeed a convex optimization problem but Theorem \ref{thm:optimality} is not a practical algorithm for finding a solution since we cannot simply choose a point and compare it with every other points in the feasible set. The work is burdensome. Therefore, we seek a way to overcome this problem and we introduce the interior-point method as our algorithm.
\section{Duality}
In this section, we briefly introduce the Lagrangian duality by defining the Lagrangian $L$, the Lagrangian multiplier and the dual function, in order to prepare for the interior-point method in the later section.\\
We consider the problem in the standard form: \begin{align*}\text{minimize} &f_0(x)\\
\text{subject to} &f_i(x)\leqslant 0, i=1,\cdots, m\\
&h_i(x)=0, i=1,\cdots, p,
\end{align*}with variable $x\in \mathbb R^n$. We assume its domain is nonempty and denote the optimal value as $p^*$. This problem is also referred as the $\textit{primal problem}$. Note the optimization problem here does not need to be convex.\\

The basic idea in duality is to take the constraints in the problem into account by incorporating into the objective function a weighted sum of the constraint functions.
\begin{definition}Define the Lagrangian $L: \mathbb R^n \times \mathbb R^m \times \mathbb R^p\rightarrow \mathbb R$ associated with the problem above as \begin{center}$L(x,\lambda,v)=f_0(x)+\sum_{i=1}^m \lambda_i f_i(x)+\sum_{i=1}^p v_i h_i(x),$\end{center}with domain $D \times \mathbb R^m \times \mathbb R^p$, with $D$ being the domain of $x$ where the object function and all constraint functions are defined. 
\end{definition}We refer to $\lambda_i$ as the $\textit{Lagrange multiplier}$ associated with the $i$th inequality constraint $f_i(x)\leqslant 0$, and $v_i$ as the Lagrange multiplier associated with the $i$th equality constraint $h_i(x)=0$. $\lambda$ and $v$ are called the $\textit{dual variables}$.
\begin{definition}The dual function $g:\mathbb R^m \times \mathbb R^p \rightarrow \mathbb R$ is the minimum value of $L(x,\lambda,v) \text{over }x\in D: \text{for }\lambda \in \mathbb R^m, v\in \mathbb R^p$, our dual variable,  \begin{center}$g(\lambda,v)=\text{inf}_{x\in D} L(x,\lambda,v)$.
\end{center}\end{definition}
Customarily, we define $d^*$ as the optimal value (the maximum value) of our dual function $g(\lambda , v)$.
\begin{theorem}The dual function $g$ yields lower bounds for the optimal value of the primal problem: \begin{center}$g(\lambda, v)\leqslant p^*.$\end{center}for any $\lambda \geqslant 0$ and any $v$.
\label{thm:dual}
\end{theorem}
The duality gap is known as the difference between $d^*$ and $p^*$.
\begin{proof}Suppose $\hat{x}$ is a feasible point for the primal problem, i,e., $\hat{x} \in D$. We then have \begin{center}$f_i(\hat{x})\leqslant 0$ and $h_i(\hat{x})=0$, and $\lambda \geqslant 0$.
\end{center}Thus \begin{center}$\sum_{i=1}^m \lambda_i f_i(\hat{x})+\sum_{i=1}^p v_i h_i(\hat{x})\leqslant 0.$\end{center}Therefore \begin{center}$L(\hat{x},\lambda,v)=f_0(\hat{x})+\sum_{i=1}^m \lambda_i f_i(\hat{x})+\sum_{i=1}^p v_i h_i(\hat{x})\leqslant f_0(\hat{x})$.\end{center}Hence
\begin{center}$g(\lambda,v)=\text{inf}_{x\in D} L(x,\lambda, v)\leqslant L(\hat{x},\lambda, v)\leqslant f_0(\hat{x})$.
\end{center}Since $g(\lambda,v)\leqslant f_0(\hat{x})$ holds for all $\hat{x}\in D$, we have $g(\lambda,v)\leqslant p^*$.
\end{proof}
This theorem will be important later in the interior-point method.
\section{Algorithms for Optimization Problems}
In this chapter we first discuss methods for solving the unconstrained convex optimization problem \begin{center}minimize $f(x)$\end{center} where $f:\mathbb R^n\rightarrow \mathbb R$ is convex and twice continuously differentiable. We will assume that the problem is solvable (there exists an optimal point $x^*$). Let the optimal value be $f(x^*)=p^*$.\\
 In the following sections we solve the problem using an iterative algorithm, an algorithm that computes a sequence of points $x^{(0)}, x^{(1)},\cdots \in \mathbb R^n$ that aims for $f(x^{(k)})\rightarrow p^*$ as $k\rightarrow \infty$. Such a sequence is called the minimizing sequence. The algorithm is terminated when $f(x^{(k+1)})-f(x^{(k)})\leqslant \epsilon$ and $f(x^{(k+2)})-f(x^{(k+1)})\leqslant \epsilon$. We want the two differences less than $\epsilon$ in a row since we want to avoid the situation when two points are accidentally close, whereas still not close to the optimal point.\\

I will first introduce the Descent method in section 6.1, which requires to find a descent direction and choose a step size for each iteration. Then in section 6.2, I will introduce the Newton's step (even it's commonly called``step," a more accurate word should be ``direction"), a specific descent search direction. Afterwards, I will present Newton's method with equality constraints that serves as the foundation of the interior-point method.
\subsection{Descent Methods}This section forms the foundation of the following two algorithms. A descent method produces a sequence $x^{(k)}$ in $\mathbb R^n, k=1,\cdots,$ where \begin{center}$x^{(k+1)}=x^{(k)}+t^{(k)}\Delta x^{(k)}$,
\end{center}where $t^{(k)}>0$ (except when $x^{(k)}$ is optimal). Here $\Delta x^{(k)}\in \mathbb R^n$ is called the $\textit{search direction}$, where $k=1,\cdots$ denotes the iteration number. Scalar $t^{(k)}$ is called the $\textit{step length}$. \\
All the methods we study are $\textit{descent methods}$, which means \begin{center}$f(x^{(k+1)})<f(x^{(k)})$.
\end{center}
\begin{prop}If $f$ is convex, then at each search step, if we want the algorithm to be descent method, then the search direction must satisfy \begin{center}$\nabla f(x^{(k)})^T \Delta x^{(k)}<0$.
\end{center}\label{prop:search}\end{prop}
\begin{proof}By theorem \ref{thm:firstorder} the first-order condition, since $f$ is convex, we know \begin{center}$f(y)\geqslant f(x)+\nabla f(x)^T(y-x)$.
\end{center}Now suppose \begin{center}$\nabla f(x^{(k)})^T (y-x^{(k)})\geqslant 0$. \end{center}We then have \begin{center}$f(y)\geqslant f(x^{(k)}).$\end{center}
If we let $f(y)=f(x^{(k+1)})$ be our new choice in the sequence of iteration, then the method is not descent anymore, which contradicts with our assumpustion for descent algorithm. Thus we must have \begin{center}$\nabla f(x^{(k)})^T \Delta x^{(k)}<0$.\end{center}
\end{proof}
Geometrically, the search direction must make an acute angle with the negative gradient. We call such a direction a $\textit{descent direction}$.\\
Algorithm of $\textit{General Descent Method}$:\\
Given a starting point $x\in \textbf{dom}f$.

Repeat\\
1. Determine a descent direction $\Delta x$.\\
2. $\textit{Line search}$. Choose a step size $t>0$.\\
3. $\textit{Update}$. $x=x+t\Delta x$.\\
Until stopping criterion is satisfied.\\
The second step is called the $\textit{line search}$ since the selection of the step size $t$ determines where along the line $\{x+t\Delta x |t\in \mathbb R_{+}\}$ the next iterate will be.
\subsection {Newton's Step}
\begin{definition}If $f$ is convex and twice differentiable, then for $x\in \textbf{dom}f$, the vector \begin{center}$\Delta x=-\nabla^2 f(x)^{-1} \nabla f(x)$\end{center}is called the $\textit{Newton step}$.\end{definition}
\begin{prop*}Newton step is a descent direction.
\end{prop*}
\begin{proof}Note that $\nabla f(x)^T \Delta x=-\nabla f(x)^T \nabla^2 f(x)^{-1} \nabla f(x)<0$ by positive definiteness of $\nabla^2 f(x)$. Then the Newton step is a descent direction by proposition \ref{prop:search}.
\end{proof}
Another perspective at why we want to choose $\Delta x$ in such a manner. By theorem \ref{thm:unconstrained} for an unconstrained problem, we know that we want \begin{center}$\nabla f_0(x^{*})=0$.
\end{center}Now since we start with $x$ and we'd like to move towards $x^{*}$, we let $v=x^{*}-x$. We linearize the optimality condition near x we get\begin{center}$\nabla f(x^{*})=\nabla f(x+v) \approx \nabla f(x)+\nabla^2 f(x)v=0$.
\end{center}The solution to $v$ in the above equation, by simple algebra, is $v=-\nabla^2 f(x)^{-1} \nabla f(x)$, our Newton step. So the Newton step is what must be added to $x$ so that the linearized optimality condition holds. This suggest that our update $x+\Delta x$ would be a good approximation of $x^{*}$.\\
The algorithm for operating Newton's Method is the same as the algorithm of general descent method, with descent direction $\Delta x$ calculated as Definition 3. 
\subsection{Newton's method with equality constraints}
In this section we describe an extension of Newton's method to include equality constraints.\\
The optimization problem now is \begin{center}minimize $f_0(x)$\\
subject to $Ax=b$,
\end{center}where $A\in \mathbb R^{m\times n}$.
We aim to derive the formula that can solve the Newton step $\Delta x$ for this problem at the feasible point $x$. By second-order Taylor approximation near $x$ for $f_0$, \begin{center}$\hat{f_0}(x+v)\approx f_0(x)+\nabla f_0(x)^T v+\frac{1}{2}v^T \nabla^2 f_0(x) v$
\end{center}with variable $v$. Therefore, at each iteration step when we have a feasible ``guess" of $x$, by finding an appropriate $v$, we can keep on minimizing $f_0$ by minimizing $\hat{f_0}(x+v)$. Thus the Newton step $\Delta x$ should be the``appropriate" $v$ for the above approximation since the Newton step $\Delta x$ is what must be added to $x$ to solve the problem (minimize $\hat{f_0}$) when the quadratic approximation is used in place of $f_0$. We then can replace the original equality constrained problem with \begin{center}minimize $f_0(x)+\nabla f_0(x)^T v+\frac{1}{2}v^T \nabla^2 f_0(x) v$\\
subject to $A(x+v)=b$.
\end{center}Note that in the above problem, $x$ is already known and $v$ is our variable. The solution $v$ to the above quadratic problem is our $\Delta x$, the Newton step at $x$. 
\begin{theorem}The Newton step is $\Delta x$ in the solution to \begin{center}$\begin{bmatrix}
\nabla^2 f_0(x) & A^T\\
A & 0
\end{bmatrix}*\begin{bmatrix}\Delta x\\
w\end{bmatrix}=\begin{bmatrix}-\nabla f_0(x)\\
0\end{bmatrix}$,
\end{center}where $w$ is some vector in $\mathbb R^m$.
\end{theorem}
In fact, $w$ happens to be the value of the optimal dual variable for the primal quadratic problem. The details of proving why $w$ is the optimal dual variable can be found in \cite{Boyd} p.522-523. Since our focus reamins to be solving $\Delta x$, and sovling $w$ is only a byproduct of this process, plus why $w$ happens to be the optimal dual variable is complicate to prove and involves lots more details in duality, we won't go into it in details.
\begin{proof}By theorem \ref{eqaulity constraint}, our optimal condition is that there exists $w\in \mathbb R^m$ such that \begin{center}$\nabla \hat{f_0}(v)+A^Tw=0$ and $Av=0.$\end{center} Note we have $Av=0$ since both $A(x+v)=Ax=b$ and thus $Av=b-b=0$. In this case, $\nabla \hat{f_0}(v)=\nabla f_0(x)+\nabla^2 f_0(x) v$. If we rearrange the equation and put the conditions into a linear system, we get the above linear system in the theorem, where $v=\Delta x$. 
\end{proof}Comparing to the previous section, Newton's step without constraints, there are two key differences: the initial point must be feasible, since we want our "guess" of $x$ to be feasible (satisfying $Ax=b$) at each iteration step; and the Newton step $\Delta x$ is a feasible direction, i.e., $A\Delta x=0$.
\section{Interior Point Algorithm}
\subsection{Interior Point Method Set Up: Logarithmic Barrier and Central Path}
In this section we discuss $\textit{interior-point methods}$ for solving convex optimization problems that include inequality constraints. \begin{align*}&\text{minimize} &f_0(x)\\
&\text{subject to} &f_i(x)\leqslant 0, i=1, \cdots, m\\
& &Ax=b,\end{align*} 
where $f_0,\cdots , f_m: \mathbb R^n \rightarrow \mathbb R$ are convex and twice continuously differentiable, and $A\in \mathbb R^{p\times n}$, with $p<n$. We assume that the problem is solvable, i.e., an optimal $x^*$ exists. Let $f_0(x^*)=p^*$.\\

In section 7.1.1 we introduce the idea of indicator function that incorporates the inequality constraint functions into the objective function. Then in 7.1.2, we develop a even better indicator function which is smooth and differentiable which approximates the original indicator function. Together with section 7.2.1 construct the preliminary interior-point method based on the previous results and the error bound. Note that this algorithm is a one-time procedure in terms of finding the search direction, instead of repeating seeking a sequence of search direcitons that take us to the optimal value (minimizing sequence). Afterwards in section 7.3, we modify the algorithm by eliminating the necessity of good starting points and moderate accuracy. However, since our current method still requires a strictly feasible starting point, we introduce the phase I method, which helps us find an intial strictly feasible starting point given any initial guess.
\subsubsection{Logarithmic barrier function and central path}
Our goal is to approximately formulate the inequality constrained problem as an equality constrained problem to which Newton's method can be applied.\\
We first rewrite the problem to incorporate the inequality constraint functions into the objective function.
\begin{center}minimize $f_0(x)+\sum_{i=1}^m I(f_i(x))$\\
subject to $Ax=b$,
\end{center}
where $I: \mathbb R\rightarrow \mathbb R$ is the indicator function: \begin{center}\begin{equation}I(u)=\begin{cases} 0 & u\leqslant 0\\
\infty & u>0.\end{cases}\end{equation}
\end{center}If we take a step back and look at this formulated problem, notice that when all of the inequality function are satisfied, i.e., $f_i(x)\leqslant 0$, the objective function boils down to the original one $f_0(x)$, otherwise, if at least one inequality function is not satisfied, the objective function becomes infinity. \\
The above problem has no inequality constraints, but its objective function is not differentiable, so Newton's method cannot be applied, nor, technically, is it really a ``function".
\subsubsection{Logarithmic barrier}
We use a better indicator function \begin{center}$\hat{I}(u)=-\frac{1}{t}log(-u),$\end{center}where $t>0$ is a parameter that sets the accuracy of the approximation to approximate the indicator function $I$. Figure 4 demonstrates this approximation.
\begin{figure}[!h]
\begin{center}\includegraphics[width=3.5 in]{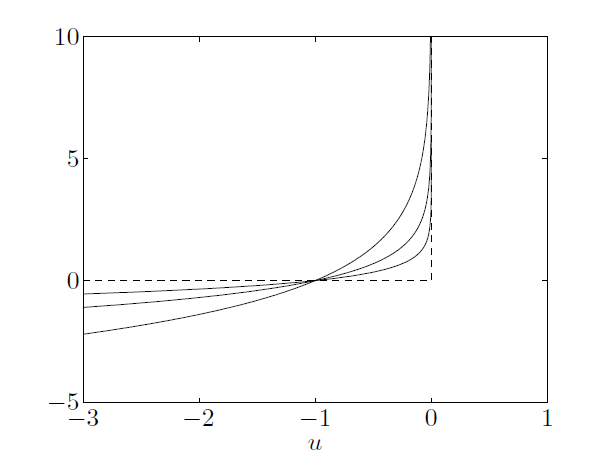}
\caption{The dashed line show the function $I(u)$, the indicator function in 7.1.1, and the solid lines show $\hat{I}(u)=-(1/t)log(-u)$, for t=0.5, 1, 2. The curve for $t=2$ shows the best approximation.}
\end{center}
\end{figure}
As we can see, $\hat{I}$ is convex (using second-order condition in $\mathbb R$) and nondecreasing, and undefined on the value $\infty$ for $u>0$. However, $\hat{I}$ is now smooth and differentiable for $u<0$. At $t$ increases, the approximation becomes more accurate.\\
The problem now becomes: \begin{center}minimize $f_0(x)+\sum_{i=1}^m -\frac{1}{t}log(-f_i(x))$\\
subject to $Ax=b$.
\end{center}Since the addition of convex functions is still convex, our objective function here is convex and now we can apply Newton's method.
\begin{definition}The function \begin{center}$\phi (x)=-\sum_{i=1}^m log(-f_i(x))$,
\end{center}with $\textbf{dom}\phi=\{x\in \mathbb R^n|f_i(x)< 0,  i=1,\cdots,m\}$ is called the logarithmic barrier for the problem.
\end{definition}Note the domain is the set of points that satisfy the inequality constraints strictly since the logarithmic barrier grows without bound if $f_i(x)\rightarrow 0_+$ for any $i$ no matter what value the positive parameter $t$ has.\\
By multivariable calculus (product rule, quotient rule and chain rule) we have \begin{align}\nabla \phi(x) = \sum_{i=1}^m \frac{1}{-f_i(x)}\nabla f_i(x)\end{align}
\begin{center}$\nabla^2 \phi (x)=\sum_{i=1}^m \frac{1}{f_i(x)^2}\nabla f_i(x) \nabla f_i(x)^T+\sum_{i=1}^m\frac{1}{-f_i(x)}\nabla^2 f_i(x)$.\end{center}
\subsubsection{Central path}
Our current optimization problem has the form \begin{center}minimize $f_0(x)+\frac{1}{t}\phi(x)$\\
subject to $Ax=b$.
\end{center}
This is equivalent to \begin{center}minimize $t f_0(x)+\phi(x)$\\
subject to $Ax=b$.
\end{center}since $t$ is a parameter and thus the problem has the same minimizers. For $t>0$, we define $x^*(t)$ as the solution of the above problem.\begin{definition}The central path is defined as the set of points $x^*(t),t>0$, which we call the central points.\end{definition}
\begin{theorem}The central points are characterized by the following necessary and sufficient conditions: $x^*(t)$ is strictly feasible, i.e., satisfies \begin{center}$A x^*(t)=b, f_i(x^*(t))<0, i=1,\cdots,m,$
\end{center} and there exists $\hat{v}\in \mathbb R^p$ such that \begin{align*}0 &= t\nabla f_0(x^*(t))+\nabla\phi(x^*(t))+A^T \hat{v}\\
&= t\nabla f_0(x^*(t))+\sum_{i=1}^m \frac{1}{-f_i(x^*(t))}\nabla f_i(x^*(t))+A^T \hat{v}
\end{align*}
\label{thm:centralpoints}
\end{theorem}
\begin{proof}The proof of the this theorem follows directly from theorem 12 and equation 7, as the optimization problem now only has equality constraints.
\end{proof}
\subsection{The preliminary barrier method}
\subsubsection{The preliminary barrier method}
The following theorem helps us define the error bound of the solution using the interior-point method. 
\begin{theorem}$f_0(x^*(t))-p^*\leqslant m/t$. \label{dual feasible}
\end{theorem}
\begin{proof}Define \begin{center}$\lambda_i^*(t)=-\frac{1}{tf_i(x^*(t))}, i=1,\cdots, m, v^*(t)=\hat{v}/t.$\end{center}
We assume that $\lambda^*(t),v^*(t)$ is dual feasible. The details of this proof, as it requires some basic knowledge of duality, will be attached in the addendum for refference.
We now calculate \begin{align*}g(\lambda^*(t), v^*(t)) &= f_0(x^*(t))+\sum_{i=1}^m \lambda_i^*(t) f_i(x^*(t))+v^*(t)^T(Ax^*(t)-b)\\
&=f_0(x^*(t))-m/t. 
\end{align*}
The last step is because the sum of the second term adds up to $-m/t$, by construction, and the third term is simply $0$.
Therefore since by theorem \ref{thm:dual} \begin{center}$g(\lambda^*(t),v^*(t))\leqslant p^*$,\end{center}we have \begin{align*}f_0(x^*(t))-p^* &\leqslant f_0(x^*(t))-g(\lambda^*(t), v^*(t))\\
&=m/t.
\end{align*}
\end{proof}
This theorem gives us a bound on the accuracy of the $x^*(t)$, the error bound is $m/t$.\\ Using the error bound, we are able to define the parameter $t$, given a specific error tolerance.

With this theorem, we can derive the barrier method with a guranteed specified accuracy $\epsilon,$ by picking $t=m/\epsilon$ and solving the equality constrained problem based on the minimization problem in section 6.3:
\begin{center}minimize  $(m/\epsilon)f_0(x)+\phi(x)$\\
subject to  $Ax=b.$
\end{center}using Newton's method with equality constraints.\\

\subsubsection{Application of the preliminary barrier method}
We will use the object function that has been proven convex in Example \ref{exa:object} as the object function in the following convex optimization problem. For the sake of simplicity, we will construct some simply constraint functions instead of those have been shown above.
\begin{example}\begin{align*}\text{minimize } & {x_1}^2+{x_2}^2+{x_3}^2-{x_1}{x_2}-{x_2}{x_3}\\
\text{subject to } &x_1+x_2\leqslant 200\\
&x_1+5 x_2+10 x_3\leqslant 8000\\
&-10x_2-x_3\leqslant 5000\\
& x_1+x_3=400\\ 
\end{align*}
\end{example}This problem is a convex optimization problem. One can check this by testing the convexity of each constraint function using the second-order conditions and since it's purely algebra, we will skip the details of calculation.\\

Before we apply the preliminary barrier method, we first calculate the answer of this problem. Note that the problem can be reduced to a two-variable optiimazation problem using the equality constraint, and then using mathematica, we can get our solution to the problem:
\begin{align*}x_1&=400/3\\
x_2 &= 200/3\\
x_3 &=800/3
\end{align*}
\begin{figure}[!h]
\begin{center}\includegraphics[width=3.5 in]{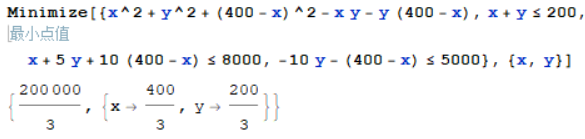}
\caption{Answer to Example 6}
\end{center}
\end{figure}
Note the solution is on the first inequality constraint while satisfying the equality constraint.\\

We then apply the preliminary barrier method for this optimization problem. In this example, our inequality constraint functions, by definition, are $f_1(x)=x_1+x_2-200$, $f_2(x)=x_1+5x_2+10x_3-8000$, and $f_3(x)=-10x_2-x_3-5000$. Thus the logartihmic barrier function is then \begin{align*}\phi(x)&=-(\log(-f_1(x))+\log(-f_2(x))+\log(-f_3(x)))\\
&= -\log(-x_1-x_2+200)-\log(-x_1-5x_2-10x_3+8000)-\log(10x_2+x_3+5000)\\.
\end{align*}
We have our equality constraint 
\[
\begin{bmatrix}
1 &0 &1\\
\end{bmatrix}
\begin{bmatrix}
x_1\\
x_2\\
x_3
\end{bmatrix}
=
\begin{bmatrix}
400\\
\end{bmatrix}.
\]
We have three inequalities and thus $m=3$, and suppose we want the accuracy to be $\epsilon= 10^{-1}$. Then according to section 7.2.1, we will then use Newton's method with equality constraints to solve the optimization problem:
\begin{center} minimize $3*10^{1} ({x_1}^2+{x_2}^2+{x_3}^2-{x_1}{x_2}-{x_2}{x_3})-\log(-x_1-x_2+200)-\log(-x_1-5x_2-10x_3+8000)-\log(10x_2+x_3+5000)$\\ \end{center}
subject to \[
\begin{bmatrix}
1 &0 &1
\end{bmatrix}
\begin{bmatrix}
x_1\\
x_2\\
x_3
\end{bmatrix}
=
\begin{bmatrix}
400\\
\end{bmatrix}.
\]
Recall that at each per iteration within the Newton's method, we solve \begin{center}$\begin{bmatrix}
\nabla^2 f(x) & A^T\\
A & 0
\end{bmatrix}\begin{bmatrix}\Delta x\\
w\end{bmatrix}=\begin{bmatrix}-\nabla f(x)\\
0\end{bmatrix}$,
\end{center} to find the direction $\Delta x$. Typically, we will use Brent's method \cite{Brent} for the Line search step, in which we choose the appropriate step size $t$ in each iteration when using Newton's method with equality constraints. Here is the result from Matlab:\\
\begin{figure}[!h]
\begin{center}\includegraphics[width=3.5 in]{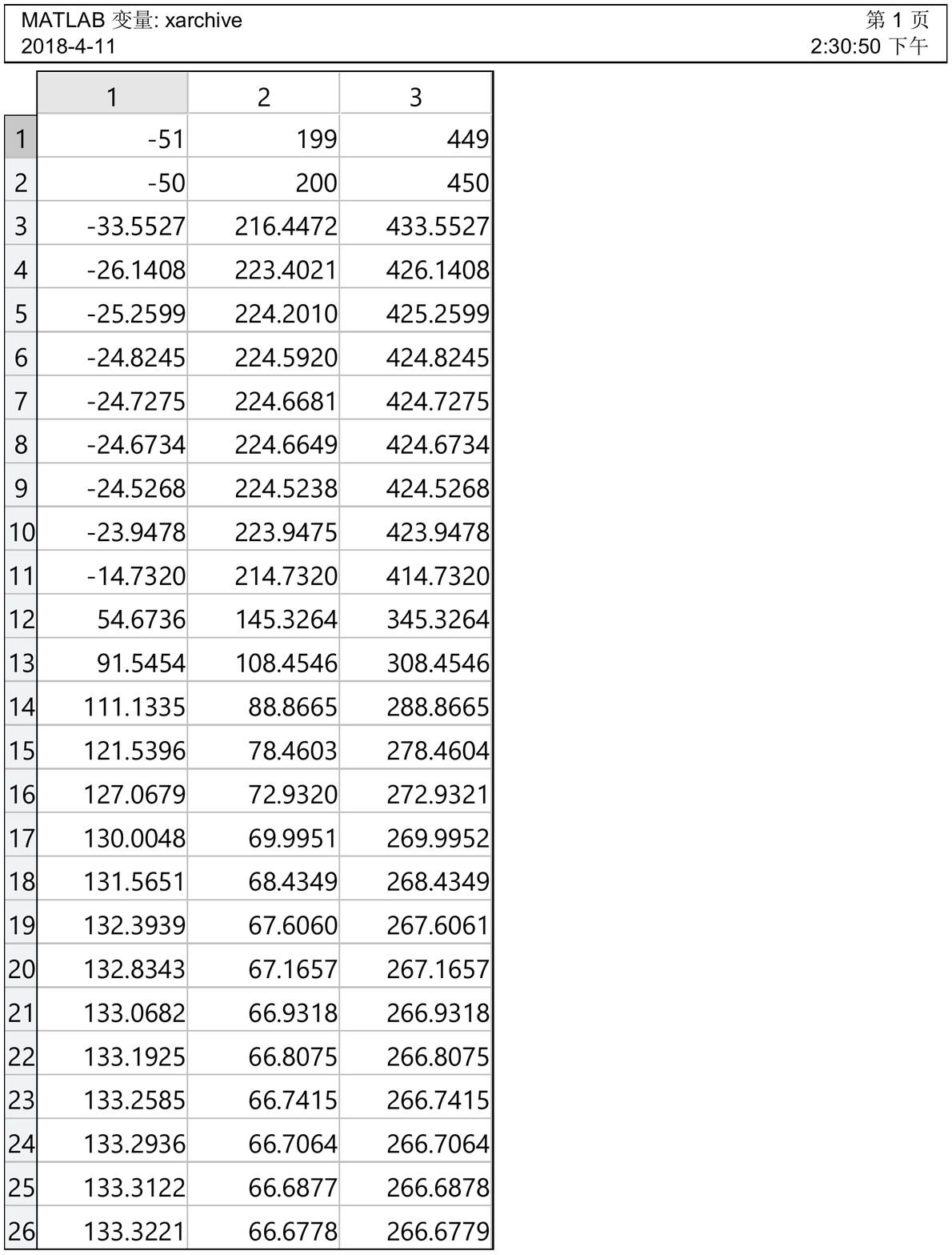}
\caption{Matlab Answer to Example 6 with preliminary barrier method, $\epsilon =10^{-1}$}
\end{center}
\end{figure}
\begin{figure}[!h]
\begin{center}\includegraphics[width=3.5 in]{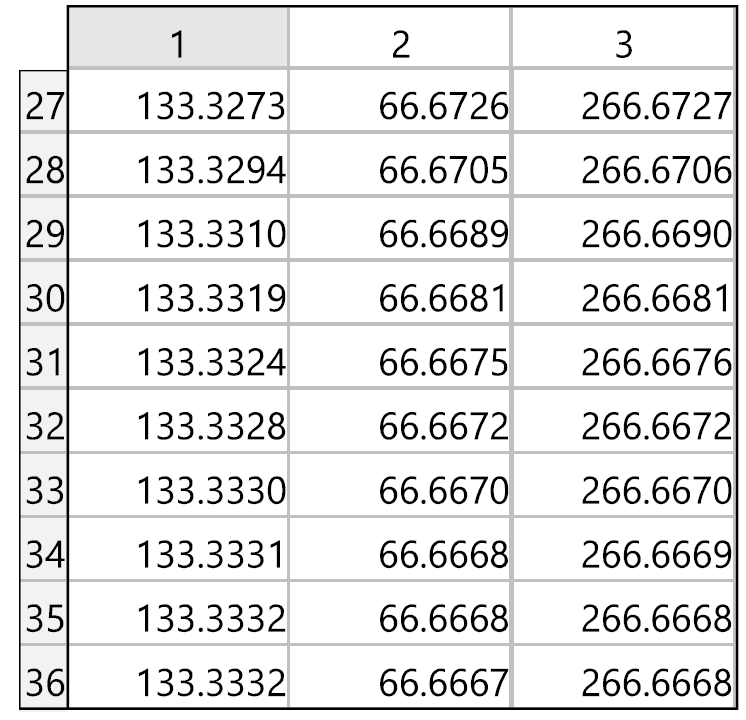}
\caption{Cont: Matlab Answer to Example 6}
\end{center}
\end{figure}
Thus we indeed get the correct numerical solution that we want after 36 iterations. We can see that throughout the iteration process, we are always heading towards the right direction, the iteration points are always strictly feasible. However, the preliminary method does not always work. If we adjust our accuracy to $10^{-10}$, figure 8 demonstrates the mal-function of this algorithm. As we can see, the solution is trapped around $\{0, 200, 400\}$, which is not the desired answer.
\begin{figure}[!h]
\begin{center}\includegraphics[width=3.5 in]{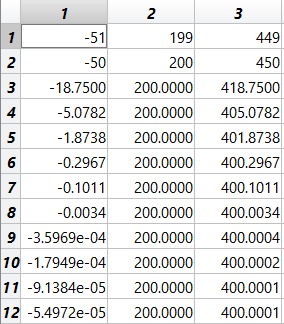}
\caption{Cont: Matlab Answer to Example 6 with accuracy $10^{-10}$}
\end{center}
\end{figure}\\

In fact, if we analyze the preliminary barrier method for solving the equality constrainted minimization problem: \begin{center}minimize  $(m/\epsilon)f_0(x)+\phi(x)$\\
subject to  $Ax=b,$
\end{center}we should see that theoretically, there are at least two drawbacks for this algorithm. First, the preliminary barrier method requires to have good starting points. The initial guess need to be strictly feasible, otherwise the logarithmic barrier function won't be defined, and when the starting point is too far away from some inequality constraint function, the logarithmic barrier function might have a huge error, as our logarithmic approximation will take it to a large negative value. Second, this algorithm requires a moderate accuracy (i.e., $\epsilon$ not too small), otherwise the term $m/\epsilon$ blows up so that $(m/\epsilon) f_0(x)$ outweights the logarithimic barrier function. 
\subsection{The Interior-Point Method: barrier method}
Because of the drawbacks discussed above, in this section, we do a simple extension on the preliminary barrier method based on solving a sequence of unconstrained minimization problems using the last point as the starting point for the next unconstrained minimization problem. Computationally, we compute $x^*(t)$ for a sequence of increasing values of $t$, until $t\geqslant m/\epsilon$, which gurantees that we have $\epsilon$ accuracy. In this case, we do not require the initial guess to be a good guess.\\

Barrier method:\\
Given strictly feasible $x, t^{(0)}>0, u>1, \epsilon >0$.\\
Repeat\\
1. $\textit{Centering step}$: Compute $x^*(t)$ by minimizing $tf_0+\phi,$ subject to $Ax=b$, starting at x.\\
2. $\textit{Update}$: $x=x^*(t)$.\\
3. $\textit{Stopping criterion}: \textbf{quit } \text{if } m/t<\epsilon$.\\
4. $\textit{Increase t}: t=ut$.\\
$\textbf{Choice of u}$:\\
The choice of $u$ involves a trade-off. If $u$ is large, after each centering step, $t$ increases a large amount, so that current iterate probably won't be a good approximation of the next iterate. Thus there would be more iteration when doing the minimization of the equality constraint optimization problem. However, on the other hand, we will reach our stopping criterion more quickly.\\
$\textbf{Choice of } t^{(0)}$:\\
If $t^{(0)}$ is large, the first centering step will require too many iterations. If $t^{(0)}$ is too small, the algorithm will need extra several iterations of the centering step. Typically, we will begin with $t^{(0)}=10$ and if the first centering step runs too many iterations, we will decrease $t^{(0)}$, until we feel comfortable with the number of iterations in the first centering step.\\
$\textbf{Convergence analysis}$:\\
It's straightforward that at each iteration, we are approaching our designated accuracy by dividing $m$ by $u$. Thus the stopping criterion we check would be $m/t, m/ut, m/u^2 t, \cdots$. Thus the duality gap after the initial centering step, and k iterations, would be $m/(u^k t^{(0)})$. Thus we acheive the desired accuracy $\epsilon$ after \begin{center}$\frac{log(m/\epsilon t^{(0)})}{log u}$\end{center} steps, plus the initial centering step.\\ Note the above idea is based on theorem 17, in which we proved the error bound at each iteration.
\subsection{The application of the Interior-point method to Example 5}
We apply the Interior-point algorithm to Example 6. Let $t^{(0)}=10$ and $u=10$. We first test the method with accuracy $\epsilon=10^{-2}$. As we can see, we indeed get the desired correct solution, and we run two Interior-point iterations (Namely the first is $3/10$, and the second one $3/10^2$). In the frist Interior-point iteration, we ran 16 Newton iterations and in the second Interior-point iteration, we ran 7 Newton iterations, as illustrated in Figure 9.
\begin{figure}[!h]
\begin{center}\includegraphics[width=3.5 in]{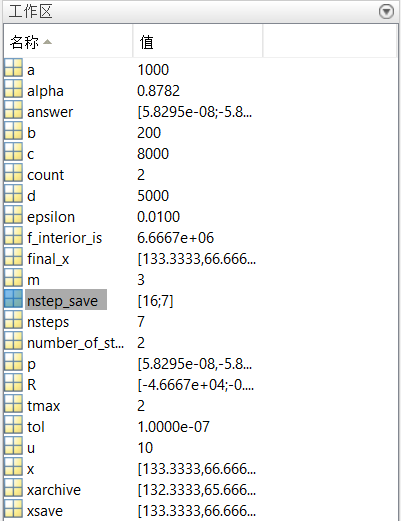}
\caption{Interior-point Algorithm for Example 6 with accuracy $10^{-2}$}
\end{center}
\end{figure}\\

We now switch to accuracy with $\epsilon=10^{-10}$, which the preliminary method cannot solve. Again, by Figure 10, we get the desired correct solution (Figure 10), and we ran 10 Interior-point iterations. The number of Newton iterations in each Interior-point iteration is respectively, 16, 7, 5, 3, 1, 1, 1, 1, 1, 1. As we can see, when $\epsilon$ is getting smaller and smaller, the answer is already pretty accurate, and thus not so many Newton iterations required in each Interior-point step.
\begin{figure}[!h]
\begin{center}\includegraphics[width=3.5 in]{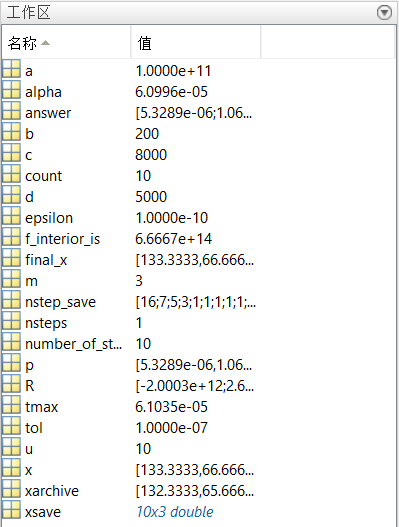}
\caption{Interior-point Algorithm for Example 6 with accuracy $10^{-10}$}
\end{center}
\end{figure}\\
\subsection{Phase I method}
The barrier method requires a strictly feasible starting point $x^{(0)}$ (it does not require it to be a good guess). When we do not have such a starting point, we need a preliminary stage, $\textit{Phase I},$ in which we derive a strictly feasible initial point.\\

We aim to find $x\in \mathbb R^n$ that satisfies the strict feasibility of the problem, i.e., \begin{center}$f_i(x)< 0, i=1,\cdots m, Ax=b$,
\end{center} with $f_i: \mathbb R^n \rightarrow \mathbb R$ convex, continuous second differentiable. Assume we have a point $x^{(0)}\in \textbf{dom} f_1 \cap \cdots \cap \textbf{dom} f_m, Ax^{(0)}=b$. \\
We form the following optimization problem: \begin{align*}\text{minimize } &s\\
\text{subject to }&f_i(x)\leqslant s, i=1,\cdots, m\\
&Ax=b
\end{align*}where $x\in \mathbb R^n, s\in \mathbb R$. Note that if we have the minimal value of $s$ less than 0, then we know that every inequality constraint is strictly below $0$, as $s$ serves as an upper bound for the inequality constraints.\\

Now, we can apply the barrier method to solve the above problem, if we are confident that this problem is strictly feasible, which is the requirement for applying the barrier method.
\begin{prop}Given $x^{(0)}$ as starting point for $x$, we are able to pick an intial guess for $s$ such that $s^{(0)}$ is strictly feasible for the problem \begin{align*}\text{minimize } &s\\
\text{subject to }&f_i(x)\leqslant s, i=1,\cdots, m\\
&Ax=b
\end{align*}
\end{prop}
\begin{proof}Assume that we are given $x^{(0)}$ as starting point for $x$. We can simply choose $s^{(0)}$ such that \begin{center}$s^{(0)}>\text{max}_{i=1,\cdots,m} f_i(x^{(0)})$.
\end{center}Then for this $s^{(0)}$, the problem must be strictly feasible since $f_i(x^{(0)})< s^{(0)}$ for all $i=1,\cdots,m$.
\end{proof}
We can thus apply the barrier method to solve the problem and this step is called the phase I optimization problem.\\

There are two cases depending on the sign of the optimal value $\bar{p}^*$ of the phase I problem.\\
1. If $\bar{p}^*<0$, then the original problem has a strictly feasible starting point $x^{(0)}$. We then just need to determine when $s<0$ for $x$ in the Phase I problem.\\
2. If $\bar{p}^*\geqslant0$, then the original problem does not have a strictly feasible starting point $x^{(0)}$ and the original problem is not feasible.\\
\section{Addendum}
The following proof is for proving  $\lambda^*(t),v^*(t)$ is dual feasible in Theorem \ref{dual feasible}
\begin{proof}Recall that $\lambda^*(t),v^*(t)$ is dual feasible if $\lambda^* >0$ and they are in the domain of the dual function $g:\mathbb R^m \times \mathbb R^p\rightarrow \mathbb R$, where the Lagrangian $L(x,\lambda,v)$ is minimized.\\
It's clear that $\lambda^*(t)>0$ because $f_i(x^*(t))<0, i=1,\cdots, m.$ Now we aim to prove that the Lagrangian is minimized.\\
By theorem \ref{thm:centralpoints}, if we divide the equation of last line by $t$ on both sides, we get \begin{align}\nabla f_0(x^*(t))+\sum_{i=1}^m \lambda_i^*(t) \nabla f_i(x^*(t))+A^Tv^*(t)=0.
\end{align}For $\lambda=\lambda^*(t)$ and $v=v^*(t)$, $x^*(t)$ is then strictly feasible and the Lagrangian \begin{center}$L=f_0(x)+\sum_{i=1}^m \lambda_i f_i(x)+v^T(Ax-b)$\end{center}is minimized by $x^*(t)$ as its first derivative is $0,$ by equation 8. Thus $\lambda^*(t),v^*(t)$ is a dual feasible pair.
\end{proof}

\end{document}